\documentclass[review,onefignum,onetabnum]{siamart250211}
\nolinenumbers

\usepackage{multirow}%
\usepackage{amsmath,amssymb,amsfonts}%
\usepackage[siam]{jphmacros2e} 
\usepackage{xcolor}%
\usepackage{textcomp}%
\usepackage{manyfoot}%
\usepackage{booktabs}%
\usepackage{algorithm}%
\usepackage{algpseudocode}%
\usepackage{listings}%
\usepackage{algorithm}
\usepackage{comment}
\usepackage{setspace}
\usepackage{float}

\usepackage{enumitem}

\usepackage{bbm}
\newtheorem{problem}{Problem}
\usepackage{lineno}

\title{Adversarial Physics-Informed Machine Learning for Robust Optimal Safe Predefined-Time Stabilization: A Game-Theoretic Approach
\thanks{Submitted to the editors \today.}
\funding{The work of N.-M.T.K and R.M. was partially supported by US DoT Safety21 National University Transportation Center and NSF under Grant CISE-$2431569$. The work of S.L., J.D., and G.E.K. was partially supported by Laboratory University Collaboration Initiative (LUCI) sponsored by the Basic Research Office, Office of Under Secretary of Defense for Research and Engineering (OUSD R\&E).}}

\author{Nick-Marios T. Kokolakis\footnotemark[2] \footnotemark[4] \ \and Shanqing Liu\footnotemark[3]\ \and Jérôme Darbon\footnotemark[3]\ \and Rahul Mangharam\footnotemark[2]\ \and George Em Karniadakis\footnotemark[3]  }

\begin{document}

\maketitle
\renewcommand{\thefootnote}{\fnsymbol{footnote}}

\footnotetext[2]{Department of Electrical and Systems Engineering, University of Pennsylvania, Philadelphia, PA $19104$ USA (\{nmkoko,  rahulm\}@seas.upenn.edu). }

\footnotetext[3]{Division of Applied Mathematics, Brown University, Providence, RI $02912$ USA (\{shanqing\_liu, jerome\_darbon, george\_karniadakis\}@brown.edu). }

\footnotetext[4]{Corresponding author.}

\begin{abstract}
We develop a game-theoretic framework for adversarially robust optimal safe prede-\ fined-time stabilization of parameter-dependent nonlinear dynamical systems with nonquadratic cost functionals. Our approach ensures that all system trajectories remain within a specified admissible set and converge to equilibrium in a predefined time despite adversarial disturbances. The control problem is formulated as a two-player zero-sum differential game, where the controller is a minimizing player and the adversary a maximizing player. We derive sufficient conditions for the existence of a saddle-point solution and safe predefined-time stability using a barrier Lyapunov function that satisfies a differential inequality and the steady-state Hamilton-Jacobi-Isaacs (HJI) equation. To address the analytical intractability of solving the HJI equation, we introduce a physics-informed learning algorithm that robustly learns the Nash safely predefined-time stabilizing control strategy. Simulation results demonstrate the efficacy and resilience of the proposed method in ensuring robust optimal safe predefined-time stabilization under adversarial disturbances.
\end{abstract}



\section{Introduction}
In control systems engineering, \textit{autonomy} pertains to controlled systems that can operate without human intervention. Systems with this capability are termed autonomous systems (ASs), including, but not limited to, self-driving cars, humanoid robots, and unmanned aerial vehicles \cite{vamvoudakis2020synchronous}. However, numerous incidents of unintended AS crashes highlight that ASs are \textit{safety-critical systems}. Hence, guaranteeing safety is crucial, yielding the emergence of \textit{safe autonomy}~\cite{xiao2023safe}. The
control systems community can enable safe autonomy
by harnessing the advantages of \textit{nonlinear control theory} \cite{haddad2011nonlinear}, \textit{optimal control theory} \cite{liberzon2011calculus}, \textit{robust control theory} \cite{bacsar2008h}, and \textit{game theory} \cite{bacsar1998dynamic} to equip ASs with control architectures that ensure safety, stability, robustness, and performance. However, enabling safe autonomy becomes considerably more challenging in the presence of adversarial disturbances, which may arise from cyber-physical attacks, and hence necessitating the design of strategic decision-making mechanisms that synthesize adversarially robust optimal safe policies guaranteeing disturbance rejection in a predefined time rather than in an infinite or finite time. By incorporating \textit{predefined-time adversarial robustness} into the safe control design, ASs can achieve enhanced resilience and reliability in real-world applications.

 \textit{Stability theory} concerns the behavior of the system trajectories of a dynamical system when the system initial state is near an equilibrium state \cite{haddad2011nonlinear}. The notion of \textit{ asymptotic stability} in dynamical systems allows the convergence of system trajectories to a Lyapunov stable equilibrium point over the infinite horizon \cite{haddad2011nonlinear}. In contrast, the concept of \textit{finite-time stability} enables the convergence of the system solutions to a Lyapunov stable equilibrium state in finite-time \cite{bhat2000finite}. An inherent drawback of finite-time stability is that the settling-time function is not uniformly bounded, and hence the time of convergence to the equilibrium point may increase (possibly unboundedly) as the vector norm of the initial condition increases. The stronger notion of \textit{fixed-time stability} ensures the convergence of the system trajectories to a finite-time stable equilibrium point in a fixed-time involving a uniformly bounded settling-time function~\cite{polyakov2011nonlinear}. A key limitation of fixed-time stability is that the upper bound of the settling-time function may not necessarily be predefined. Alternatively, the concept of \textit{predefined-time stability} \cite{sanchez2018class,jimenez2020lyapunov} involves fixed-time stable \textit{parameter-dependent} dynamical systems whose upper bound of the settling-time function can be chosen via an appropriate selection of the system parameters. 
 
 \textit{Optimal control theory} concerns the synthesis of a control law for a given dynamical system to optimize a user-defined cost functional \cite{liberzon2011calculus}. The notions of \textit{optimality} and \textit{stability} are intertwined in \textit{optimal feedback stabilization} control problems that involve the synthesis of feedback control laws that guarantee closed-loop system stability while optimizing a given cost functional. The problems of optimal asymptotic, finite-time, fixed-time, and predefined-time stabilization are addressed in \cite{bernstein1993nonquadratic,haddad2015finite, kokolakis2023fixed,jimenez2017optimal}. Although these optimal control frameworks provide stability and performance guarantees, they do not account for adversarial disturbances, that is, strategic inputs designed to degrade performance or destabilize the system. Alternatively, \textit{differential game theory} analyzes strategic interactions between competing agents, providing a powerful framework for addressing robust optimal control problems. Specifically, the disturbance rejection control problem in $\mathcal{H}_\infty$ control theory \cite{ball1989h,van1993nonlinear,bacsar2008h} can be formulated as a \textit{two-player zero-sum game} wherein the controller is a minimizing player and the disturbance is a maximizing player. Infinite-horizon zero-sum games for linear and nonlinear
dynamical systems with quadratic and nonquadratic cost functionals are studied in \cite{jacobson1977values,mageirou1976values,limebeer1992game,l2016differential,l2017differentialCTA,l2017differentialJFI,l2017differential}. In particular, sufficient conditions are provided to characterize a saddle point solution for the game and closed-loop system asymptotic, partial-state asymptotic,  and partial-state finite-time stability in \cite{l2016differential,l2017differentialCTA,l2017differentialJFI}. In light of the above, although these game-theoretic control frameworks establish stability and Nash equilibrium, predefined-time stability is \textit{not} ensured, and safety is \textit{not} a design consideration.

\textit{Safe control theory} concerns the controller analysis and synthesis for a safety critical dynamical system to guarantee the satisfaction of safety specifications \cite{cohen2023adaptive}, which can be expressed as forward invariance \cite{haddad2011nonlinear} of a set of safe system states. \textit{Control barrier functions} (CBF) have been commonly used to guarantee the safety of a control system by rendering a safe set forward invariant \cite{wieland2007constructive,ames2016control,ames2019control}. The problem of \textit{asymptotic} stabilization with guaranteed safety is addressed in \cite{romdlony2016stabilization} by merging a control Lyapunov function (CLF) \cite{sontag1989universal} and a CBF \cite{wieland2007constructive}. However, a mapping that is both a CBF and a CLF
%CLBF 
cannot exist, as shown in \cite{braun2020comment}. Alternatively, quadratic programming (QP) has been utilized to combine a CLF and a CBF to synthesize controllers for safe \textit{asymptotic} stabilization of nonlinear systems \cite{ren2022razumikhin,mestres2022optimization}. A key limitation of these frameworks is the lack of robustness guarantees. In contrast, building on the results of \cite{xu2015robustness}, the notion of robust CBF is introduced in \cite{jankovic2018robust} to construct controllers for nonlinear systems with exogenous disturbances that guarantee safety and input-to-state
stability \cite{haddad2011nonlinear}. However, CBF-based QPs introduce undesirable \textit{asymptotically} stable equilibria \cite{reis2020control},  do not optimize closed-loop system performance in the face of a worst-case adversary, and do not enforce time constraints. To the best of our knowledge, a game-theoretic control framework that \textit{simultaneously} ensures \textit{safety}, \textit{predefined-time stability}, \textit{optimality}, and \textit{robustness} is absent from the literature. 

To synthesize the Nash equilibrium strategies for a zero-sum differential game formulating a robust optimal stabilization problem, it is first necessary to determine the value of the game, which is a \textit{stabilizing} solution of a nonlinear partial differential equation, the steady-state
Hamilton–Jacobi–Isaacs (HJI) equation. Solving the steady-state HJI equation is challenging, except for special cases \cite{lewis2012optimal}.  \textit{Physics-informed neural networks} (PINNs), first developed in \cite{raissi2019physics}, have demonstrated their effectiveness in approximating a stabilizing solution to the steady-state Hamilton–Jacobi–Bellman (HJB) equation related to an \textit{optimal stabilization problem }\cite{furfaro2022physics,fotiadis2023physics,kokolakis2025safe}. However, using PINNs
to solve a zero-sum differential game has \textit{not} been explored. To the best of our knowledge, the literature lacks a game-theoretic physics-informed learning framework that approximates the unique stabilizing solution to the steady-state HJI.

\paragraph*{Contributions} The contributions of this paper are fourfold.
First, we address an adversarially robust optimal safe predefined-time stabilization problem formulated as a two-player zero-sum differential game. Second, sufficient conditions for the existence of a saddle-point solution to the zero-sum game and closed-loop system safe predefined time stability are derived in terms of a barrier Lyapunov function. Third, a robust \textit{inverse} optimal safe predefined time stabilization problem is addressed. Finally, an \textit{adversarial physics-informed learning} algorithm is developed to learn the solution to the robust optimal safe predefined-time stabilization problem. A block diagram of the proposed adversarial PINN control framework is shown in Fig~\ref{fig:sketch}.
\begin{figure}
    \centering
    \includegraphics[width=0.75\linewidth]{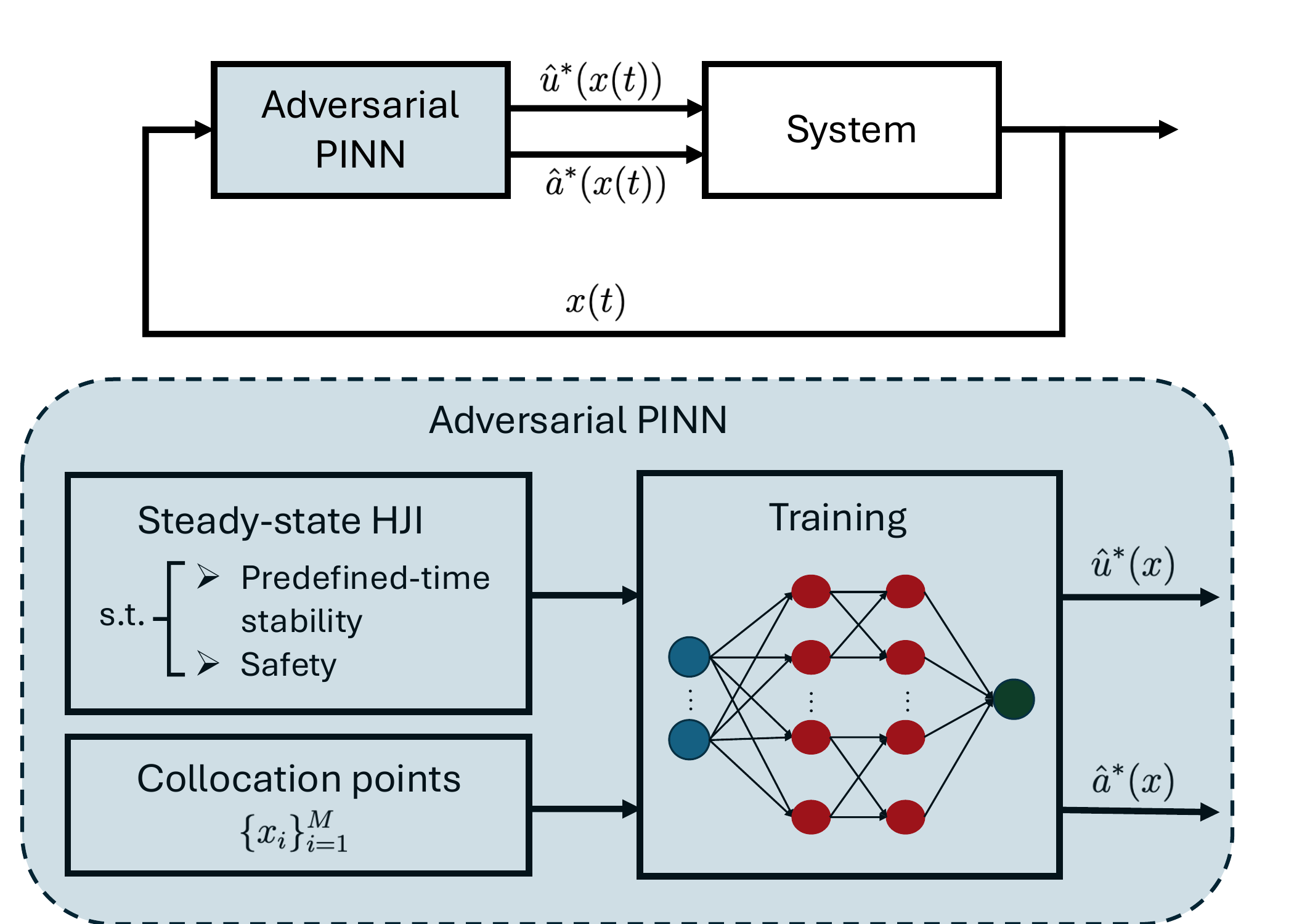}
    \caption{Block diagram of the adversarial PINN control framework for robust optimal safe predefined-time stabilization. The adversarial PINN is trained \textit{offline} using a set of collocation points to solve the steady-state HJI equation subject to predefined-time stability and safety constraints. The learned Nash control strategy $\hat{u}^\star(\cdot)$ is then applied in feedback to ensure robust optimal safe predefined-time stabilization under the adversarial disturbance $\hat{a}^\star(\cdot)$.}
    \label{fig:sketch}
\end{figure}
\paragraph*{Structure} The remainder of the paper is structured as follows. The notion of safe predefined-time stability for general parameter-dependent nonlinear dynamical systems is introduced in Section \ref{sec:Safe Predefined-Time Stability}, while Section \ref{sec: robust optimal safe predefined-time stabilization problem} states the robust optimal safe predefined-time stabilization problem. Section \ref{sec: optimal and inverse optimal affine} introduces the robust optimal and inverse optimal safe predefined-time stabilization problem for parameter-dependent nonlinear \textit{affine} dynamical systems. In Section \ref{sec:PINN}, an adversarial physics-informed learning algorithm is developed to learn the solution to the robust optimal safe predefined-time stabilization problem. Section \ref{sec:Simulation Results} presents illustrative numerical examples. Finally, conclusions are provided in Section \ref{sec:Future work}.

\paragraph*{Notation}
The notation used in this paper is standard. Specifically, $\left\|\cdot\right\|_p\triangleq\left[\sum_{i=1}^{n}\left|x_{i}\right|^{p}\right]^{1 / p},\  1$ $\leq  p <\infty,$ denotes the $\ell^p$-norm of a vector. We interchangeably use the notation $V^{\prime}(x)$ and $V_x(x)$ to denote the gradient of a scalar-valued function $V(x)$ with respect to a vector-valued variable $x$, which is defined as a row vector. The signum function $\operatorname{sgn}: \mathbb{R} \rightarrow\{-1,0,1\}$ is defined as $\operatorname{sgn} \left(x\right) \triangleq x /|x|,\ x \neq 0$, and $\operatorname{sgn}(0) \triangleq 0$. We define the gamma function $\Gamma(\cdot)$ as $\Gamma(x) \triangleq  \int_0^{\infty} e^{-t} t^{x-1} \mathrm{d} t,\ x>0$. The indicator function of a set $A \subseteq \mathbb{R}^{n}$ is the function $ \mathbbm{1}_A: \mathbb{R}^{n} \rightarrow \{0,1\}$ defined by $ \mathbbm{1}_A(x) \triangleq 1,\ x \in A$, and  $ \mathbbm{1}_A(x) \triangleq 0,\ x \notin A$. Let $\lceil\cdot\rfloor^{\eta}\triangleq|\cdot|^{\eta} \operatorname{sgn}(\cdot)$, where $|\cdot|$ and $\operatorname{sgn}(\cdot)$ operate component-wise and $\eta>0$. The distance of a point $x_{0} \in \mathbb{R}^{n}$ to a closed set $C \subseteq \mathbb{R}^{n}$ in the norm $\|\cdot\|$ is defined as $ \operatorname{dist}\left(x_{0}, C\right)\triangleq\inf_{x \in C} \left\{\left\|x_{0}-x\right\| \right\}$. Finally, the notation $\partial \mathcal{S}$  denotes the boundary of the set $\mathcal{S}$.

 \section{Safe Predefined-Time Stability} \label{sec:Safe Predefined-Time Stability}
In this section, we introduce the concept of \textit{safe predefined-time stability} to characterize a class of nonlinear dynamical systems with the property that every trajectory starting in a given set of admissible states containing an equilibrium point remains in the set of admissible states and converges to the equilibrium point in a predefined time. Furthermore, we give sufficient conditions for safe predefined-time stability in terms of a barrier Lyapunov function.

To introduce the notions of \textit{finite-} and \textit{fixed-time stability}, consider the \textit{parameter-independent} nonlinear dynamical system given by
\begin{equation} \label{eqn_x_aynamics_controlled  ind}
\dot{x}(t) = f(x(t)),\quad x(0) = x_0,\quad t \geq 0, 
\end{equation}
where $x(t) \in \mathcal{D} \subseteq \mathbb{R}^n,\ t \geq 0,$ is a system state vector,
$\mathcal{D}$ is an open set with $0 \in \mathcal{D}$, $f : \mathcal{D}  \to \mathbb{R}^n$
is continuous on $\mathcal{D}$ and  satisfies $f(0) = 0$.

\begin{definition} [\cite{bhat2000finite}]
The zero solution $x(t)\equiv 0$ to \eqref{eqn_x_aynamics_controlled  ind} is \emph{finite-time stable} if it is Lyapunov stable and finite-time convergent, i.e., for all $x(0) \in \mathcal{N} \backslash\{0\}$, where $\mathcal{N}\subseteq \mathcal{D}$ is an open neighborhood of the origin, $\lim _{t \rightarrow T(x(0))} x(t)=0$, where $T(\cdot)$ is the settling-time function such that $T(x(0))<\infty,\  x(0) \in \mathcal{N}$. The zero solution $x(t)\equiv 0$ to \eqref{eqn_x_aynamics_controlled  ind} is \emph{globally finite-time stable} if it is finite-time stable with $\mathcal{N}=\mathcal{D}=\mathbb{R}^{n}$.\frqed
\end{definition}

\begin{definition} [\cite{polyakov2011nonlinear}]
The zero solution $x(t)\equiv 0$ to \eqref{eqn_x_aynamics_controlled  ind} is \emph{fixed-time stable} if it is finite-time stable and the settling-time function $T(\cdot)$ is uniformly bounded, i.e., there exists $T_{\max }>0$ such that $T\left(x(0)\right) \leq T_{\max },\ x(0) \in \mathcal{N}$. The zero solution $x(t)\equiv 0$ to \eqref{eqn_x_aynamics_controlled  ind} is \emph{globally fixed-time stable} if it is fixed-time stable with $\mathcal{N}=\mathcal{D}=\mathbb{R}^{n}$.\frqed
\end{definition}
The main difference between finite- and fixed-time stability is that in finite-time stability the settling-time function is not necessarily uniformly bounded, whereas fixed-time stability involves a uniformly bounded settling-time function.
 
Alternatively, to introduce the concept of \textit{predefined-time stability}, consider the \textit{parameter-dependent} nonlinear dynamical system given by
\begin{equation} \label{eqn_x_aynamics_controlled  undist}
\dot{x}(t) = f(x(t),\theta),\quad x(0) = x_0,\quad t \geq 0, 
\end{equation}
where, for every $t \geq 0$, $x(t) \in \mathcal{D} \subseteq \mathbb{R}^n$ is a system state vector,
$\mathcal{D}$ is an open set with $0 \in \mathcal{D}$, $\theta \in \mathbb{R}^N$ is a system parameter vector, $f : \mathcal{D}  \times \mathbb{R}^{N} \to \mathbb{R}^n$
 is such that $f(\cdot,\theta)$ is continuous on $\mathcal{D}$ for all $\theta \in \mathbb{R}^{N}$ and $f(0,\cdot) = 0$. We write $s\left(t, x_0, \theta \right),\ t \geq 0$, to denote the solution to \eqref{eqn_x_aynamics_controlled  undist} with initial condition $x_0$ and system parameter $\theta$.

\begin{definition}[\cite{jimenez2020lyapunov}]
The zero solution $x(t)\equiv 0$ to \eqref{eqn_x_aynamics_controlled  undist} is \emph{predefined-time stable} with a predefined time $T_{\mathrm{p}}>0$ if 
there exists a system parameter vector $\theta \in \mathbb{R}^N$ such that the zero solution $x(t)\equiv 0$ to \eqref{eqn_x_aynamics_controlled  undist} is fixed-time stable with the settling-time function $T(\cdot,\theta)$ being uniformly bounded by $T_{\mathrm{p}}$, i.e., $T\left(x(0),\theta \right) \leq T_{\mathrm{p}},\ x(0) \in \mathcal{N}$. The zero solution $x(t)\equiv 0$ to \eqref{eqn_x_aynamics_controlled  undist} is \emph{globally predefined-time stable} with a predefined time $T_{\mathrm{p}}>0$ if it is predefined stable with a predefined time $T_{\mathrm{p}}>0$ with $\mathcal{N}=\mathcal{D}=\mathbb{R}^{n}$.\frqed
\end{definition}

\begin{remark}
    Note that the system parameter vector $\theta$ implicitly depends on the predefined time $T_{\mathrm{p}}$. \frqed
\end{remark}

The key difference between fixed- and predefined-time stability is that in  pre-\ defined-time stability the upper bound of the settling-time function is defined a priori as an explicit function of the system parameters, whereas in fixed-time stability it is not necessarily predefined. It is important to note that fixed-time stability concerns the stability of a parameter-independent nonlinear system, unlike predefined-time stability, which involves parameter-dependent nonlinear systems. Lyapunov theorems for finite-, fixed-, and predefined-time stability are provided in \cite{bhat2000finite,polyakov2011nonlinear,jimenez2020lyapunov}.

A  set of \textit{admissible states} $\mathcal{S} \subset \mathcal{D}$ is called a \textit{safe set} with respect to the parameter-dependent nonlinear dynamical system \eqref{eqn_x_aynamics_controlled  undist} if there exists  a system parameter  $\theta \in \mathbb{R}^{N}$ such that, for every $x_0 \in \mathcal{S}$, the solution $s\left(t, x_0, \theta\right),\ t \geq 0$, to \eqref{eqn_x_aynamics_controlled  undist} satisfies $\operatorname{dist}\left(s\left(t, x_0, \theta\right), \mathcal{S}\right)\equiv 0$. Equivalently, the set $\mathcal{S}$ is safe if there exists a system parameter vector such that $\mathcal{S}$  is positively invariant with respect to \eqref{eqn_x_aynamics_controlled  undist}.

The next definition introduces the concept of \textit{safe predefined-time stability}.
\begin{definition} [\cite{kokolakis2024safe}]
 Let $\mathcal{S} \subset \mathcal{D}$ be a set of admissible states with $0 \in \mathcal{S}$ and let $T_{\mathrm{p}} >0$ be a predefined time. The zero solution $x(t)\equiv 0$ to \eqref{eqn_x_aynamics_controlled  undist} is \emph{safely predefined-time stable} with predefined time $T_{\mathrm{p}}$ with respect to the set of admissible states $\mathcal{S}$  if there exists  a system parameter $\theta \in \mathbb{R}^{N}$ such that, for every $x_0 \in \mathcal{S}$, the solution $s\left(t, x_0, \theta \right),\ t \geq 0$, to \eqref{eqn_x_aynamics_controlled  undist} satisfies $\operatorname{dist}\left(s\left(t, x_0, \theta \right), \mathcal{S}\right)\equiv 0$ and $s\left(t, x_0, \theta \right)=0$ for all $t \geq T_{\mathrm{p}}$.\frqed
\end{definition}

\begin{remark}
Note that safe predefined-time stability unifies safety and predefined-time stability since positive invariance of the set of admissible states and predefined-time stability of the origin are established. Furthermore, note that the system parameter vector $\theta$ implicitly depends on the predefined time $T_{\mathrm{p}}$ and the set of admissible states $\mathcal{S}$. \frqed
\end{remark}

To present the main result of this section, the key definition of a \textit{coercive} function is needed.
\begin{definition} [\cite{bertsekasnonlinear}]
    Let $\mathcal{S} \subseteq \mathbb{R}^n $ be an unbounded set. A function $V: \mathcal{S} \rightarrow \mathbb{R}$ is called \emph{coercive} if, for every sequence $\left\{x_k\right\} _{k=1}^{\infty}$ in $\mathcal{S} $ such that $\lim _{k \rightarrow \infty} \left\|x_k\right\| = \infty$, we have $\lim _{k \rightarrow \infty} V\left(x_k\right)=\infty$.\frqed
\end{definition}

The following key theorem gives sufficient conditions for safe predefined time stability of a parameter-dependent nonlinear dynamical system.
\begin{theorem} [\cite{kokolakis2025safe}]  \label{thm: Lyapunov safe predefined time stability}Consider the  parameter-dependent nonlinear dynamical system  \eqref{eqn_x_aynamics_controlled  undist}. Let  $\mathcal{S} \subset \mathcal{D}$ be a set of admissible states with $0 \in \mathcal{S}$ and let $T_{\mathrm{p}} >0$ be a predefined time. Assume that there exist a continuously differentiable function $V: \mathcal{S} \rightarrow \mathbb{R}$, a system parameter vector $\theta \in \mathbb{R}^{N}$, and real numbers $\alpha,\ \beta,\ p,\ q,\ k >0$ such that $p k<1,\ q k>1$, and 
\begin{align}
 V(0)&=0, \label{eq:them 1 v1}\\
 V(x)&>0, \quad x \in \mathcal{S}\backslash\{0\},\label{eq:them 1 v2}\\
 V(x) & \rightarrow \infty \text { as } x \rightarrow \partial \mathcal{S},\label{eq:them 1 v3}\\
\hspace*{-0.2cm} V^{\prime}(x) f(x,\theta) &\leq -\frac{\gamma}{T_{\mathrm{p}}}\Big(\alpha V^p(x)+\beta V^q(x)\Big)^r, \quad x \in \mathcal{S} \label{eq:them 1 vdot},
\end{align}
where 
\begin{align*}
\gamma \triangleq \frac{\Gamma\left(\frac{1-k p}{q-p}\right) \Gamma\left(\frac{k q-1}{q-p}\right)}{\alpha^k \Gamma(k)(q-p)}\left(\frac{\alpha}{\beta}\right)^{\frac{1-k p}{q-p}}.
\end{align*}
 If either $\mathcal{S}$ is bounded or both  $\mathcal{S}$ is unbounded and $V(\cdot)$ is coercive, then the zero solution $x(t) \equiv 0$ to \eqref{eqn_x_aynamics_controlled  undist}  is safely predefined time stable with predefined time $T_{\mathrm{p}}$ with respect to the set of admissible states $\mathcal{S}$.
\end{theorem} 
A continuously differentiable function $V(\cdot)$ satisfying \eqref{eq:them 1 v1}-\eqref{eq:them 1 v3} is called a \textit{safely predefined-time stabilizing Lyapunov function candidate} for the nonlinear dynamical system \eqref{eqn_x_aynamics_controlled  undist}. If, additionally, $V(\cdot)$ satisfies \eqref{eq:them 1 vdot}, $V(\cdot)$ is called a \textit{safely predefined-time stabilizing Lyapunov function} for the nonlinear dynamical system \eqref{eqn_x_aynamics_controlled  undist}.

\begin{remark} Theorem \ref{thm: Lyapunov safe predefined time stability} investigates safe predefined-time stability for the nonlinear dynamical system \eqref{eqn_x_aynamics_controlled  undist} without requiring knowledge of the system trajectories. \frqed
\end{remark}

 \section{Robust Optimal Safe Predefined-Time Stabilization} \label{sec: robust optimal safe predefined-time stabilization problem}
 
In this section, we formulate the robust optimal safe predefined time stabilization problem as a \textit{two-player zero-sum differential game} to characterize robust optimal feedback controllers that render the equilibrium point of the closed-loop system safely predefined time while optimizing the closed-loop system performance against the worst-case feedback adversary. Specifically, we provide sufficient conditions for the existence of a safely predefined time stabilizing feedback \textit{Nash equilibrium} solution to the zero-sum game.

 Consider the controlled parameter-dependent nonlinear dynamical system given by
\begin{align} \label{eq: general system}
\dot{x}(t)= & F(x(t),\theta_\mathrm{s}, u(t), a(t)),\quad x(0)=x_{0},\quad  t\geq 0,
\end{align}
where, for every $t \geq 0$, $x(t) \in \mathcal{D} \subseteq \mathbb{R}^n$ is the state vector,
$\mathcal{D}$ is an open set with $0 \in \mathcal{D}$,  $\theta_\mathrm{s} \in \mathbb{R}^{N_\mathrm{s} }$ is the system parameter vector,
$u(t) \in U \subseteq \mathbb{R}^{m_u}$ is the control input with $0 \in U$,
$a(t) \in A \subseteq \mathbb{R}^{m_a}$ is the adversary input with $0 \in A$,
and $F : \mathcal{D} \times \mathbb{R}^{N_\mathrm{s} } \times U \times A \to \mathbb{R}^n$
is such that $ F(\cdot, \theta_\mathrm{s},\cdot,\cdot)$  is jointly continuous on $\mathcal{D} \times U \times A $ for all  $\theta_\mathrm{s} \in \mathbb{R}^{N_\mathrm{s} }$ and $F(0,\cdot,0,0) = 0$. The control $u(\cdot)$ and adversary $a(\cdot)$ in \eqref{eq: general system} belong to the class of \textit{admissible controls} $\mathcal{U} \triangleq\left\{u: [0, \infty) \rightarrow U: u(\cdot)\right.$ is Lebesgue measurable$\}$ and \textit{admissible adversaries} $\mathcal{A} \triangleq\left\{a: [0, \infty) \rightarrow A: a(\cdot)\right.$ is Lebesgue measurable$\}$. We assume that the required properties for the existence and uniqueness of solutions to \eqref{eq: general system} are satisfied, and we write $s\left(t, x_0,\theta_\mathrm{s}, u(\cdot), a(\cdot)\right), \ t\geq 0 $, to denote the solution to \eqref{eq: general system} with  initial condition $x_0$, system parameter  $\theta_\mathrm{s}$, admissible control $u(\cdot)$, and admissible adversary $a(\cdot)$.

The mappings $u^\star: \mathcal{S} \times \mathbb{R}^{N_\mathrm{c}}\to U$ and $a^\star: \mathcal{S} \times \mathbb{R}^{N_\mathrm{a}}\to A$ such that $u^\star(\cdot,\theta_\mathrm{c} )$ and $a^\star(\cdot,\theta_\mathrm{a} )$ are Lebesgue measurable functions on $\mathcal{S}$ for every control parameter vector $ \theta_\mathrm{c} \in \mathbb{R}^{N_\mathrm{c}}$ and adversary parameter vector $ \theta_\mathrm{a} \in \mathbb{R}^{N_\mathrm{a}}$
  with  $u^\star(0,\cdot) = 0$ and $a^\star(0,\cdot) = 0$ are called a \textit{control strategy} and an \textit{adversary strategy}. Furthermore, if $u(t) = u^\star(x(t),\theta_\mathrm{c}),\ t \geq 0$, and $a(t) = a^\star(x(t),\theta_\mathrm{a}),\ t \geq 0$, where
$u^\star(\cdot,\cdot)$ is a control strategy and $a^\star(\cdot,\cdot)$ is an adversary strategy and $x(t)$ is a solution to \eqref{eq: general system}, then we call $u(\cdot)$ a \textit{feedback control strategy} and $a(\cdot)$ a \textit{feedback adversary strategy}. Note that a feedback control strategy is an admissible control and a feedback adversary strategy is an admissible adversary since $u^\star(\cdot,\theta_\mathrm{c} )$ and $a^\star(\cdot,\theta_\mathrm{a} )$ are Lebesgue measurable functions and take values in $U$ and $A$. 

Given a control strategy $u^\star(\cdot,\cdot)$ and an adversary strategy $a^\star(\cdot,\cdot)$ and a feedback control strategy
$u(t) = u^\star(x(t),\theta_\mathrm{c}),\ t \geq 0$, and a feedback adversary strategy $a(t) = a^\star(x(t),\theta_\mathrm{a}),\ t \geq 0$, the closed-loop system is given by
\begin{align}\label{closed-loop}
\dot{x}(t) =  &F(x(t),\theta_\mathrm{s},u^\star(x(t),\theta_\mathrm{c}),a^\star(x(t),\theta_\mathrm{a})),\quad  x(0) = x_0,  \quad  t  \geq 0, 
\end{align} 
which can be cast in the form of \eqref{eqn_x_aynamics_controlled  undist} with 
$N=N_\mathrm{s}+N_\mathrm{c}+N_\mathrm{a}$, $\theta=[\theta_\mathrm{s}^\mathrm{T},\ \theta_\mathrm{c}^\mathrm{T},\ \theta_\mathrm{a} ^\mathrm{T}]^\mathrm{T}$, and
$f(x,\theta)=F\left(x,\theta_\mathrm{s},u^\star\left(x,\theta_\mathrm{c}\right), a^\star\left(x,\theta_\mathrm{a}\right)\right)$.

We now define the notion of a \textit{safely predefined-time stabilizing  pair of feedback strategies}.
\begin{definition}
Consider the controlled parameter-dependent nonlinear dynamical system 
given by \eqref{eq: general system}. Let   $\mathcal{S} \subset \mathcal{D}$ be a set of admissible states with $0 \in \mathcal{S}$ and let $T_{\mathrm{p}} >0$ be a predefined time. The pair $\left(u^\star(\cdot), a^\star(\cdot) \right)$
of the feedback control strategy $u(\cdot)=u^\star(x(\cdot),\theta_\mathrm{c})$ and the feedback adversary strategy $a(\cdot)=a^\star(x(\cdot),\theta_\mathrm{a})$ is \emph{safely predefined-time stabilizing} if there exists a system parameter $\theta_\mathrm{s} \in \mathbb{R}^{N_\mathrm{s} } $  such that 
the zero solution $x(t) \equiv 0$ to the closed-loop system \eqref{closed-loop} is safely predefined-time stable with predefined time $T_{\mathrm{p}}$ with respect to the set of admissible states $\mathcal{S}$. \frqed
\end{definition}
Given a set of admissible states  $\mathcal{S} \subset \mathcal{D}$ with $0 \in \mathcal{S}$ and a predefined time $T_{\mathrm{p}} >0$, we define, for every $ x_0 \in \mathcal{S}$, the set of \textit{safely predefined-time stabilizing pairs} of feedback strategies by $\mathcal{F}\left( x_0,\mathcal{S},T_{\mathrm{p}}\right)\triangleq\{u: [0, \infty) \rightarrow U$ and $a: [0, \infty) \rightarrow A: u(\cdot)$ and $a(\cdot)$ are a feedback control strategy and a feedback adversary strategy and $s\left(t, x_0,\theta_\mathrm{s}, u(\cdot), a(\cdot)\right), \ t\geq 0 $, is a solution to \eqref{eq: general system} satisfying $\operatorname{dist}(s\left(t, x_0,\theta_\mathrm{s}, u(\cdot), a(\cdot) \right),$ $ \mathcal{S}  )\equiv 0$ and $s\left(t, x_0,\theta_\mathrm{s}, u(\cdot), a(\cdot)\right)=0$ for all $t \geq T_{\mathrm{p}}\} \subset \mathcal{U} \times \mathcal{A}$. Furthermore, let $\mathcal{F}_u\left( x_0,\mathcal{S},T_{\mathrm{p}}\right)\subset \mathcal{U}$ and $\mathcal{F}_a\left( x_0,\mathcal{S},T_{\mathrm{p}}\right)\subset \mathcal{A}$ be the set of feedback control strategies and feedback adversary strategies such that $\mathcal{F}_u\left( x_0,\mathcal{S},T_{\mathrm{p}}\right) \times \mathcal{F}_a\left( x_0,\mathcal{S},T_{\mathrm{p}}\right) =\mathcal{F}\left( x_0,\mathcal{S},T_{\mathrm{p}}\right)$.

To evaluate the performance of the controlled parameter-dependent nonlinear dynamical system \eqref{eq: general system} over the time interval $[0, T_{\mathrm{p}}]$ for a given predefined time $T_{\mathrm{p}} >0$ and a set of admissible states $\mathcal{S}$  with $0 \in \mathcal{S}$, we define, for every $x_0 \in \mathcal{S},\ \theta_\mathrm{s} \in \mathbb{R}^{N_\mathrm{s} }$, $ u(\cdot) \in \mathcal{U} $, and $ a(\cdot) \in \mathcal{A}$,  the cost functional 
\begin{align}\label{eqn_generic_performance_measure}
J\left(x_{0},\theta_\mathrm{s}, u(\cdot), a(\cdot)\right) \triangleq \int_{0}^{T_{\mathrm{p}}} r(x(t), u(t), a(t)) \mathrm{d} t,
\end{align}
where $r: \mathcal{S}  \times U \times A \to \mathbb{R}$ is jointly continuous in $x$, $u$, and $a$. The controller $u(\cdot)$ is a player that \textit{minimizes} the cost functional \eqref{eqn_generic_performance_measure}, while the adversary $a(\cdot)$ is a player that \textit{maximizes} the cost functional \eqref{eqn_generic_performance_measure}.

In light of the above, we now state the robust optimal safe predefined time stabilization problem as a \textit{two-player zero-sum game}.
\begin{problem}
Consider the controlled parameter-dependent nonlinear dynamical system 
given by \eqref{eq: general system} with the cost functional \eqref{eqn_generic_performance_measure}. Let the controller $u(\cdot)$ be a player that \textit{minimizes} the cost functional \eqref{eqn_generic_performance_measure} and let the adversary $a(\cdot)$ be a player that \textit{maximizes} the cost functional \eqref{eqn_generic_performance_measure}. Let  $\mathcal{S} \subset \mathcal{D}$ be a set of admissible states with $0 \in \mathcal{S}$ and let $T_{\mathrm{p}} >0$ be a predefined time. For every $ x_0 \in \mathcal{S}$, let $\mathcal{F}\left( x_0,\mathcal{S},T_{\mathrm{p}} \right)\subset \mathcal{U} \times \mathcal{A}$ be the set of safely predefined time stabilizing pairs of feedback strategies and suppose that $\mathcal{F}\left( x_0,\mathcal{S},T_{\mathrm{p}}\right)$ is nonempty. For every initial condition $x_0 \in \mathcal{S}$, synthesize $(u^\star(\cdot), a^\star(\cdot)) \in \mathcal{F}\left( x_0,\mathcal{S},T_{\mathrm{p}} \right)$ such that the zero solution $x(t) \equiv 0$ to the closed-loop system \eqref{closed-loop} is safely predefined-time stable with predefined time $T_{\mathrm{p}}$ with respect to the set of admissible states $\mathcal{S}$ and $(u^\star(\cdot), a^\star(\cdot)) $ is a saddle point of the cost functional \eqref{eqn_generic_performance_measure}. \frqed 
\end{problem}
In other words, the robust optimal safe predefined-time stabilization problem involves synthesizing a feedback control strategy $u^\star(\cdot)$ to ensure safe predefined-time stability while optimizing performance against the worst-case feedback adversary strategy $a^\star(\cdot)$ so that $(u^\star(\cdot), a^\star(\cdot))$ is a saddle point of the cost functional \eqref{eqn_generic_performance_measure}.

For every $x_{0} \in \mathcal{S}$ and $\theta_\mathrm{s} \in \mathbb{R}^{N_\mathrm{s} }$, the robust optimal safe predefined-time stabilization problem involves the minimax control problem
\begin{align*}
\min _{u(\cdot) \in \mathcal{F}_u\left( x_0,\mathcal{S},T_{\mathrm{p}}        \right)} \max _{a(\cdot) \in \mathcal{F}_a\left( x_0,\mathcal{S},T_{\mathrm{p}}        \right)} J\left(x_{0},\theta_\mathrm{s}, u(\cdot), a(\cdot)\right)
\end{align*}
subject to \eqref{eq: general system}. 

The next theorem gives sufficient conditions for the existence of a safely predefined time stabilizing feedback \textit{Nash equilibrium} solution to the two-player zero-sum game. For the statement of this result, we need to define the Hamiltonian function 
\begin{align}\label{eqn_hamiltonian}
H(x,\theta_\mathrm{s},u,a,\lambda)\triangleq r(x, u, a) + \lambda^{ \mathrm{T}} F(x, \theta_\mathrm{s},u,a),& \nonumber \\ (x,\theta_\mathrm{s},u,a,\lambda)\in \mathcal{S}\times \mathbb{R}^{N_\mathrm{s} }  \times U \times A \times \mathbb{R}^{n }.&
\end{align}
\begin{theorem} \label{Thm: general Nash equilibrium}
Consider the controlled parameter-dependent nonlinear dynamical system 
given by \eqref{eq: general system} with cost functional \eqref{eqn_generic_performance_measure}. Let the controller $u(\cdot)$ be a player that \textit{minimizes} the cost functional \eqref{eqn_generic_performance_measure} and let the adversary $a(\cdot)$ be a player that \textit{maximizes} the cost functional \eqref{eqn_generic_performance_measure}. Let $\mathcal{S} \subset \mathcal{D}$ be a set of admissible states with $0 \in \mathcal{S}$ and let $T_{\mathrm{p}} >0$ be a predefined time. For every $ x_0 \in \mathcal{S}$, let $\mathcal{F}\left( x_0,\mathcal{S},T_{\mathrm{p}} \right)\subset \mathcal{U} \times \mathcal{A}$ be the set of safely predefined time stabilizing pairs of feedback strategies and suppose that $\mathcal{F}\left( x_0,\mathcal{S},T_{\mathrm{p}}\right)$ is nonempty.  Assume that there exist a continuously differentiable function $V : \mathcal{S} \to \mathbb{R}$, a system parameter vector $\theta_\mathrm{s}\in \mathbb{R}^{N_\mathrm{s} } $, real numbers $\alpha,\ \beta,\ p,\ q,\ r>0$ such that $p r<1$ and $q r>1$, a feedback control strategy $u^\star: \mathcal{S} \times \mathbb{R}^{N_\mathrm{c}}\to U$, and a feedback adversary strategy $a^\star: \mathcal{S} \times \mathbb{R}^{N_\mathrm{a}}\to A$
such that
\begin{gather}
 H(x,\theta_\mathrm{s},u^\star(x,\theta_\mathrm{c} ), a^\star(x,\theta_\mathrm{a} ),V^{\prime \mathrm{T}}(x))= 0, \quad  x \in \mathcal{S},\label{thm H=0} \\
 H(x,\theta_\mathrm{s},u^\star(x,\theta_\mathrm{c}),a,V^{\prime \mathrm{T}}(x))\leq 0, \quad  (x, a) \in \mathcal{S} \times A, \label{thm H<0} \\
H(x,\theta_\mathrm{s},u,a^\star(x,\theta_\mathrm{a}),V^{\prime \mathrm{T}}(x))\geq 0, \quad (x, u)\in \mathcal{S} \times U, \label{thm H>0}\\
u^\star(0,\theta_\mathrm{c} ) = 0,\label{thm ustar} \\
a^\star(0,\theta_\mathrm{a} ) = 0, \label{thm astar} \\
 V(0)=0, \label{eq:them 2 lyap1}\\
 V(x)>0, \quad x \in \mathcal{S}\backslash\{0\},\label{eq:them 2 lyap2}\\
 V(x)  \rightarrow \infty \text { as } x \rightarrow \partial \mathcal{S},\label{eq:them2 lyap3}\\
 V^{\prime}(x) F(x,\theta_\mathrm{s},u^\star(x,\theta_\mathrm{c}),a^\star(x,\theta_\mathrm{a})) %\nonumber \\ 
 \leq -\frac{\gamma}{T_{\mathrm{p}}}\Big(\alpha V^p(x)+\beta V^q(x)\Big)^r,\quad  x \in \mathcal{S}, \label{eq:them 2 lyap4}
\end{gather}
where 
\begin{align*}
\gamma \triangleq \frac{\Gamma\left(\frac{1-r p}{q-p}\right) \Gamma\left(\frac{r q-1}{q-p}\right)}{\alpha^r \Gamma(r)(q-p)}\left(\frac{\alpha}{\beta}\right)^{\frac{1-r p}{q-p}}.
\end{align*}
 If  either $\mathcal{S}$ is bounded or both  $\mathcal{S}$ is unbounded and $V(\cdot)$ is coercive, then with the feedback control strategy $u(\cdot)=  u^\star(x(\cdot),\theta_\mathrm{c}) $ and the feedback adversary strategy $a(\cdot)= a^\star(x(\cdot),\theta_\mathrm{a})$, the zero solution $x(t) \equiv 0$ to \eqref{closed-loop} is safely predefined-time stable with  predefined time $T_{\mathrm{p}}$ with respect to the set of admissible states $\mathcal{S}$. Furthermore, if $x_0 \in \mathcal{S}$, then the 
pair of feedback strategies $\left(u^{\star}(\cdot), d^{\star}(\cdot)\right)$ is the Nash equilibrium of the two-player zero-sum game in the sense that
\begin{align} 
J(x_0,\theta_\mathrm{s},u^\star(x(\cdot),\theta_\mathrm{c}),a^\star(x(\cdot),\theta_\mathrm{a})) 
& = \min _{u(\cdot) \in \mathcal{F}_u\left( x_0,\mathcal{S},T_{\mathrm{p}}        \right)} \max _{a(\cdot) \in \mathcal{F}_a\left( x_0,\mathcal{S},T_{\mathrm{p}}        \right)} J\left(x_{0},\theta_\mathrm{s}, u(\cdot), a(\cdot)\right) \nonumber \\ & = \max _{a(\cdot) \in \mathcal{F}_a\left( x_0,\mathcal{S},T_{\mathrm{p}}        \right)} \min _{u(\cdot) \in \mathcal{F}_u\left( x_0,\mathcal{S},T_{\mathrm{p}}        \right)} J\left(x_{0},\theta_\mathrm{s}, u(\cdot), a(\cdot)\right) \label{eq:minimax}
\end{align}
and the Nash value is
\begin{equation}
J(x_0,\theta_\mathrm{s},u^\star(x(\cdot),\theta_\mathrm{c}),a^\star(x(\cdot),\theta_\mathrm{a}))  = V(x_0). \label{eqn_cost_hjb}
\end{equation} 
\end{theorem}
\begin{proof}
Safe predefined-time stability is a direct consequence
of \eqref {eq:them 2 lyap1}–\eqref{eq:them2 lyap3} and \eqref{eq:them 2 lyap4} by applying Theorem \ref{thm: Lyapunov safe predefined time stability} to the closed-loop
system \eqref{closed-loop}.

Next, to prove Nash equilibrium, let $x_0 \in \mathcal{S}$, let $u(\cdot) =u^\star(x(\cdot),\theta_\mathrm{c})$, let $a(\cdot) =a^\star(x(\cdot),\theta_\mathrm{a})$, and let $x(t)=s\left(t, x_0,\theta_\mathrm{s},u^\star(x(\cdot),\theta_\mathrm{c}),a^\star(x(\cdot),\theta_\mathrm{a})\right), \ t\geq 0 $, be the solution to \eqref{closed-loop}. Then, since the time derivative of $V(\cdot)$ along the trajectories of the closed-loop
system \eqref{closed-loop} is  defined by
\begin{align*}
\dot{V}(x(t)) \triangleq V^{\prime}(x(t)) F(x(t),\theta_\mathrm{s},u^\star(x(t),\theta_\mathrm{c}),a^\star(x(t),\theta_\mathrm{a})), \quad t \geq 0,
\end{align*}
and safe predefined-time stability of \eqref{closed-loop} implies $  x(t) \in \mathcal{S},\ t\geq 0$, it follows from \eqref{thm H=0} that
\begin{align} \label{proorf:hjb}
     H(x(t),\theta_\mathrm{s},& u^\star(x(t),\theta_\mathrm{c}),a^\star(x(t),\theta_\mathrm{a}),V^{\prime \mathrm{T}}(x(t))) \nonumber \\
      &=r(x, u^\star(x(t),\theta_\mathrm{c}),a^\star(x(t),\theta_\mathrm{a})) +\dot{V}(x(t))
     \nonumber \\  &=0,  \quad x(t) \in \mathcal{S},\quad t\geq 0.
\end{align}
Now, integrating \eqref{proorf:hjb} over the time interval $[0, T_{\mathrm{p}}]$ and using the definition of the cost functional \eqref{eqn_generic_performance_measure} yields
\begin{align*}
J\left(x_0,\theta_\mathrm{s},u^\star(x(\cdot),\theta_\mathrm{c}),a^\star(x(\cdot),\theta_\mathrm{a})\right) +V(x(T_{\mathrm{p}}))=V\left(x_0\right),
\end{align*}
which, since $  V(x(T_{\mathrm{p}}))=0$ by \eqref{eq:them 2 lyap1} and safe predefined-time stability of \eqref{closed-loop}, implies \eqref{eqn_cost_hjb}.

Next, let $x_0 \in \mathcal{S}$, let $u(\cdot) =u^\star(x(\cdot),\theta_\mathrm{c})$, let $a(\cdot) \in \mathcal{F}_a\left( x_0,\mathcal{S},T_{\mathrm{p}} \right) $, and let \\ $x(t)=s\left(t, x_0,\theta_\mathrm{s},u^\star(x(\cdot),\theta_\mathrm{c}), a(x(\cdot),\theta_\mathrm{a})\right),\ t \geq 0 $, be the solution to 
\begin{align} \label{closed loop ustar d}
\dot{x}(t)=  F(x(t),\theta_\mathrm{s}, u^\star(x(t),\theta_\mathrm{c}), a(x(t),\theta_\mathrm{a})),\quad x(0)=x_{0},\quad  t\geq 0.
\end{align} 
Then, since the time derivative of $V(\cdot)$ along the trajectories  of \eqref{closed loop ustar d} is given by
\begin{align*}
\dot{V}(x(t)) \triangleq V^{\prime}(x(t))F(x(t),\theta_\mathrm{s}, u^\star(x(t),\theta_\mathrm{c}), a(x(t),\theta_\mathrm{a})), \quad t \geq 0,
\end{align*}
and $(u^\star(\cdot), a(\cdot)) \in \mathcal{F}\left( x_0,\mathcal{S},T_{\mathrm{p}} \right) $ implies $  x(t) \in \mathcal{S},\ t\geq 0$, 
it follows from \eqref{thm H<0}  that
\begin{align} \label{proorf:hjb ineq} 
   H(x(t),\theta_\mathrm{s},u^\star(x(t),& \theta_\mathrm{c}),d(x(t),\theta_\mathrm{a}),V^{\prime \mathrm{T}}(x(t)))  \nonumber\\
 &  =r(x(t), u^\star(x(t),\theta_\mathrm{c}),d(x(t),\theta_\mathrm{a})) +\dot{V}(x(t))
     \nonumber \\
 & \leq  0,  \quad x(t) \in \mathcal{S},\quad t\geq 0.
\end{align}
Now, integrating \eqref{proorf:hjb ineq} over  $[0, T_{\mathrm{p}}]$, using \eqref{eqn_generic_performance_measure}, \eqref{eq:them 2 lyap1}, and  \eqref{eqn_cost_hjb}, and using the fact that \\ $(u^\star(\cdot), a(\cdot)) \in \mathcal{F} \left( x_0,\mathcal{S},T_{\mathrm{p}} \right) $  yields
\begin{align} \label{eq:left}
J\left(x_0,\theta_\mathrm{s},u^\star(\cdot), a(\cdot)\right) \leq &V\left(x_0\right) \nonumber \\
=& J\left(x_0,\theta_\mathrm{s},u^\star(\cdot), a^\star(\cdot)\right). 
\end{align}

Next, let $x_0 \in \mathcal{S}$, let $u(\cdot) \in \mathcal{F}_u\left( x_0,\mathcal{S},T_{\mathrm{p}} \right) $, let $a(\cdot) =a^\star(x(\cdot),\theta_\mathrm{a})$, and let \\ $x(t)=s\left(t, x_0,\theta_\mathrm{s},u(x(\cdot),\theta_\mathrm{c}), a^\star(x(\cdot),\theta_\mathrm{a})\right), \ t\geq 0 $, be the solution to 
\begin{align} \label{closed loop u dstar}
\dot{x}(t)=  F(x(t),\theta_\mathrm{s}, u(x(t),\theta_\mathrm{c}), a^\star(x(t),\theta_\mathrm{a})),\quad x(0)=x_{0},\quad  t\geq 0.
\end{align} 
Then, computing the time derivative of $V(\cdot)$ along the trajectories of \eqref{closed loop u dstar} as
\begin{align*}
\dot{V}(x(t)) \triangleq V^{\prime}(x(t))F(x(t),\theta_\mathrm{s}, u(x(t),\theta_\mathrm{c}), a^\star(x(t),\theta_\mathrm{a})), \quad t \geq 0,
\end{align*}
and using the fact that $(u(\cdot), a^\star(\cdot)) \in \mathcal{F}\left( x_0,\mathcal{S},T_{\mathrm{p}} \right) $ implies $  x(t) \in \mathcal{S},\ t\geq 0$, 
it follows from \eqref{thm H>0}  that
\begin{align} \label{proorf:hjb ineq >0} 
  H(x(t),\theta_\mathrm{s},&u(x(t),\theta_\mathrm{c}),a^\star(x(t),\theta_\mathrm{a}),V^{\prime \mathrm{T}}(x(t))) \nonumber \\ 
 & =r(x(t), u(x(t),\theta_\mathrm{c}),a^\star(x(t),\theta_\mathrm{a})) +\dot{V}(x(t))
     \nonumber \\ & \geq  0,  \quad x(t) \in \mathcal{S},\quad t\geq 0.
\end{align}
Now, integrating \eqref{proorf:hjb ineq >0} over  $[0, T_{\mathrm{p}}]$ and using \eqref{eqn_generic_performance_measure}, \eqref{eq:them 2 lyap1}, \eqref{eqn_cost_hjb}, and the fact that $(u(\cdot), a^\star(\cdot)) \in \mathcal{F}\left( x_0,\mathcal{S},T_{\mathrm{p}} \right) $, we obtain
\begin{align} \label{eq:right}
J\left(x_0,\theta_\mathrm{s},u(\cdot), a^\star(\cdot)\right) \geq &V\left(x_0\right) \nonumber \\ 
=& J\left(x_0,\theta_\mathrm{s},u^\star(\cdot), a^\star(\cdot)\right). 
\end{align}
In light of \eqref{eq:left} and \eqref{eq:right}, it follows that, for every $x_0 \in \mathcal{S},\ u(\cdot) \in \mathcal{F}_u( x_0,\mathcal{S},T_{\mathrm{p}} )$, and $a(\cdot) \in \mathcal{F}_a\left( x_0,\mathcal{S},T_{\mathrm{p}} \right)$,
\begin{align*}
    J\left(x_0,\theta_\mathrm{s},u^\star(\cdot), a(\cdot)\right) & \leq J\left(x_0,\theta_\mathrm{s},u^\star(\cdot), a^\star(\cdot)\right) \\& \leq J\left(x_0,\theta_\mathrm{s},u(\cdot), a^\star(\cdot)\right),
\end{align*}
which implies \eqref{eq:minimax}.    
\frQED
\end{proof}

It is important to note that, unlike QP-based methods that depend on system trajectories, Theorem \ref{Thm: general Nash equilibrium} examines robust optimal safe predefined-time stabilization for the the controlled parameter-dependent nonlinear dynamical system \eqref{eq: general system} with the cost functional \eqref{eqn_generic_performance_measure} without knowledge of the system trajectories.
\begin{remark}
The following observations provide insights into the  robust optimal safe predefined-time stabilization problem.

      1) Equation \eqref{eq:minimax} asserts that, for all $x_0 \in \mathcal{S}$, $\left(u^{\star}(\cdot), d^{\star}(\cdot)\right)$ is a saddle point equilibrium for the two-player zero sum game over the set of safely predefined time stabilizing pairs of feedback strategies $\mathcal{F}\left( x_0,\mathcal{S}, T_{\mathrm{p}}\right)$, where no player benefits from a unilateral deviation from the equilibrium. Note that an explicit characterization of  $\mathcal{F}_u\left( x_0,\mathcal{S},T_{\mathrm{p}}\right),\ \mathcal{F}_a\left( x_0,\mathcal{S},T_{\mathrm{p}}\right),$ and  $\mathcal{F}\left( x_0,\mathcal{S},T_{\mathrm{p}}\right)$ is not required.
      
      2) Conditions \eqref{eq:them 2 lyap1}-\eqref{eq:them 2 lyap4} guarantee safe predefined time stability, while conditions \eqref{thm H=0}-\eqref{thm H>0} establish Nash equilibrium over the set of admissible control values $U$ and adversary values $A$. 

      3) Equation \eqref{thm H=0}  is the steady-state HJI equation for the controlled parameter-dependent nonlinear dynamical system \eqref{eq: general system} with the cost functional \eqref{eqn_generic_performance_measure}. Furthermore, note that \eqref{thm H=0}, \eqref{thm H<0}, and \eqref{thm H>0} imply
\begin{align*}
& \underset{u \in U }{ \min } \ \underset{a \in A }{ \max } \  H(x,\theta_\mathrm{s},u,a,V^{\prime \mathrm{T}}(x)) %\\& 
= \underset{a \in A }{ \max }  \ \underset{u \in U }{ \min }  \  H(x,\theta_\mathrm{s},u,a,V^{\prime \mathrm{T}}(x))%\\&
=0, \quad  x \in \mathcal{S},
\end{align*}
which is an alternative expression for the steady-state  Hamilton-Jacobi-Isaacs equation \eqref{thm H=0}.

      4) The Nash feedback control strategy $u^\star(x,\theta_\mathrm{c})$ is an \textit{adversarially robust optimal control} that minimizes the maximum value of the Hamiltonian function \eqref{eqn_hamiltonian} for every $x \in \mathcal{S}$ with respect to $U$, that is, 
      \begin{align*}
      u^\star(x,\theta_\mathrm{c})&=\underset{u \in U}{\arg \min }\   \underset{a \in A }{ \max } \ H(x,\theta_\mathrm{s},u,a,V^{\prime \mathrm{T}}(x)) \\ 
      &=\underset{u \in U}{\arg \min }  \   H(x,\theta_\mathrm{s},u,a^\star(x,\theta_\mathrm{a}),V^{\prime \mathrm{T}}(x)). 
     \end{align*}
      Hence, it follows that $u^\star(\cdot,\theta_\mathrm{c})$ is the optimal response against the worst-case adversary $a^\star(\cdot,\theta_\mathrm{a})$ regardless of the initial condition $x_{0} \in \mathcal{S}$. 

     5) Equation \eqref{eqn_cost_hjb} shows that the safely predefined time stabilizing Lyapunov function $V(\cdot)$ for the closed-loop system \eqref{closed-loop} serves as a value of the game quantifying the best worst-case performance or, equivalently, the worst best-case performance.

     6) Although an explicit expression for the settling-time function $T(\cdot,\cdot)$ cannot be given, safe predefined time stability guarantees that there exists a parameter vector  $\theta \in \mathbb{R}^{N}$  such that  the settling-time function $T(\cdot,\theta)$ is uniformly bounded by $T_{\mathrm{p}}$, that is, $\sup _{x_0 \in \mathcal{S} } T\left(x_0,\theta \right) \leq T_{\mathrm{p}} $.    
     
     7) The system parameter vector $\theta_\mathrm{s}\in \mathbb{R}^{N_\mathrm{s}}$, the control parameter vector $ \theta_\mathrm{c} \in \mathbb{R}^{N_\mathrm{c}}$, and the adversary parameter vector $ \theta_\mathrm{a} \in \mathbb{R}^{N_\mathrm{a}}$ implicitly depend on the predefined time $T_{\mathrm{p}}$ and the set of admissible states $\mathcal{S}$.\frqed
 \end{remark}

\section{Robust Optimal and Inverse Optimal Safe Predefined-Time Stabilization  for Nonlinear Affine Systems} \label{sec: optimal and inverse optimal affine}
In this section, we specialize the results of Section \ref{sec: robust optimal safe predefined-time stabilization problem} to parameter-dependent nonlinear affine dynamical systems of the form
\begin{align} \label{eq: affine system}
\dot{x}(t)=  f(x(t),\theta_f)+G(x(t),& \theta_G)u(t) +K(x(t),\theta_K)a(t), %\nonumber \\ & 
\quad x(0)=x_{0},\quad  t\geq 0,
\end{align}
where, for every $t \geq 0$, $x(t) \in \mathbb{R}^n$ is the state vector, $u(t) \in \mathbb{R}^{m_u}$ is the control input,  $a(t) \in \mathbb{R}^{m_a}$ is the adversary input,  $f: \mathbb{R}^{n} \times \mathbb{R}^{N_f }\rightarrow \mathbb{R}^{n}$ is such that $ f(\cdot, \theta_f)$  is  continuous on $\mathbb{R}^{n}$ for all  $\theta_f \in \mathbb{R}^{N_f }$   and  $f(0,\cdot) = 0$, $G: \mathbb{R}^{n} \times \mathbb{R}^{N_G } \rightarrow \mathbb{R}^{n \times m_u}$  is such that $ G(\cdot, \theta_G)$ is  continuous on $\mathbb{R}^{n}$ for all  $\theta_G \in \mathbb{R}^{N_G }$, and $K: \mathbb{R}^{n} \times \mathbb{R}^{N_K} \rightarrow \mathbb{R}^{n \times m_a}$  is such that $ K(\cdot, \theta_K)$ is continuous on $\mathbb{R}^{n}$ for all  $\theta_K \in \mathbb{R}^{N_K }$. 
In this case, the system parameter vector $\theta_\mathrm{s} \in \mathbb{R}^{N_\mathrm{s}}$ is defined as $\theta_\mathrm{s} \triangleq [ \theta_f^\mathrm{T},\ \theta_G ^\mathrm{T},\ \theta_K^\mathrm{T}]^\mathrm{T}$.

For every $(x,u,a) \in \mathcal{S}  \times \mathbb{R}^{m_u}\times \mathbb{R}^{m_a}$, we consider running costs $r(x,u,a)$ of the form
\begin{align} \label{eq: performance integrand}
r(x, u, a)\triangleq & L(x)+L_{u}(x) u+L_{a}(x)a+u^{\mathrm{T}} R_u(x) u %\nonumber \\ & 
-a^{\mathrm{T}} R_a(x) a ,
\end{align}
where $L: \mathcal{S} \rightarrow \mathbb{R},\  L_u: \mathcal{S} \rightarrow \mathbb{R}^{1 \times m_u}, \  L_a: \mathcal{S} \rightarrow \mathbb{R}^{1 \times m_a},\ R_u: \mathcal{S} \rightarrow \mathbb{R}^{m_u \times m_u}$, and $R_a: \mathcal{S} \rightarrow \mathbb{R}^{m_a \times m_a}$ are continuous on $\mathcal{S}$ such that $L(0)=0$, $L_{u}(0)=0$, $L_{a}(0)=0$, $R_u(x) \succ 0,\ x \in \mathcal{S}$, and $R_a(x) \succ 0,\ x \in \mathcal{S}$. Note that $L(\cdot)$ penalizes deviations of the state $x$ from the origin, if $L(x) \geq 0,\ x \in \mathcal{S}$, $L_u(\cdot)$ penalizes alignment with the control input $u$, $L_a(\cdot)$ rewards alignment with the adversary input $a$, and $R_u(\cdot)$ and $R_a(\cdot)$ penalize control and adversary effort.

\begin{remark}
The functions $L_u(x)$ and $L_a(x)$ in the running cost \eqref{eq: performance integrand} are \textit{arbitrary vector functions} of $x \in \mathcal{S}$ and provide the cross-weighting terms $L_u(x)u$ and $L_a(x)a$ in the running cost $r(x,u,a)$. Note that the terms $L_u(x)u$ and $L_a(x)a$ can be written as $
 L_{u}(x)u=\left\| L_{u}^\mathrm{T}(x) \right\|_2 \left\|u\right\|_2  \cos (\phi_{ L_{u}}(x)-\phi_u),\  x \in \mathcal{S},
$ and \ $ L_{a}(x)a=\left\| L_{a}^\mathrm{T}(x) \right\|_2 \left\|d\right\|_2  \cos (\phi_{ L_{a}}(x)-\phi_a),\  x \in \mathcal{S},$ where $\phi_{ L_{u}}(x)-\phi_u$ is the angular difference between $L_{u}^\mathrm{T}(x)$ and $u$ at $ x \in \mathcal{S}$ and $\phi_{L_{a}}(x)-\phi_a$ is the angular difference between $L_{a}^\mathrm{T}(x)$ and $a$ at $ x \in \mathcal{S}$. Hence, $ L_{u}(x)$ penalizes alignment with the control input $u$ at $ x \in \mathcal{S}$ and $ L_{a}(x)$ rewards alignment with the adversary input $a$ at $ x \in \mathcal{S}$. For example, $L_u(\cdot)$ and $L_a(\cdot)$ may be designed to indicate the control and adversary input directions that yield instability, unsafety, or low performance.\frqed   
\end{remark}

For every $x_0 \in \mathcal{S},\ \theta_\mathrm{s} \in \mathbb{R}^{N_\mathrm{s} }$, $ u(\cdot) \in \mathcal{U} $, and $ a(\cdot) \in \mathcal{A}$, the cost functional \eqref{eqn_generic_performance_measure}, with $r(\cdot, \cdot, \cdot)$ given by \eqref{eq: performance integrand}, becomes 
\begin{align}  \label{eq:performance measure}
J\left(x_0,\theta_\mathrm{s},u(\cdot),a(\cdot)\right) =&\int_{0}^{T_{\mathrm{p}}} \Big(L(x(t))+L_{u}(x(t)) u(t) +L_{a}(x(t)) a(t) \nonumber\\  &
+u^{\mathrm{T}}(t) R_u(x(t))u(t) -a^{\mathrm{T}}(t) R_a(x(t)) a(t) \Big)\mathrm{d} t.
\end{align}
Next, we specialize Theorem \ref{Thm: general Nash equilibrium} to parameter-dependent nonlinear affine dynamical systems \eqref{eq: affine system} with the cost functional \eqref{eq:performance measure}.
\begin{corollary}\label{Thm: affine Nash equilibrium}
Consider the parameter-dependent nonlinear affine dynamical
system \eqref{eq: affine system} with cost functional \eqref{eqn_generic_performance_measure}. Let the controller $u(\cdot)$ be a player that \textit{minimizes} the cost functional \eqref{eqn_generic_performance_measure} and let the adversary $a(\cdot)$ be a player that \textit{maximizes} the cost functional \eqref{eqn_generic_performance_measure}. Let $\mathcal{S} \subset  \mathbb{R}^{n}$ be a set of admissible states with $0 \in \mathcal{S}$ and let $T_{\mathrm{p}} >0$ be a predefined time. For every $ x_0 \in \mathcal{S}$, let $\mathcal{F}\left( x_0,\mathcal{S},T_{\mathrm{p}} \right)\subset \mathcal{U} \times \mathcal{D}$ be the set of safely predefined time stabilizing pairs of feedback strategies and suppose that $\mathcal{F}\left( x_0,\mathcal{S},T_{\mathrm{p}}\right)$ is nonempty. Assume that there exist a continuously differentiable function $V : \mathcal{S} \to \mathbb{R}$, system parameter vectors $\theta_f\in \mathbb{R}^{N_f }$, $\theta_G\in \mathbb{R}^{N_G }$, and $\theta_K \in \mathbb{R}^{N_K}$, and real numbers $\alpha,\ \beta,\ p,\ q,\ r>0$ such that $p r<1,\ q r>1$, and 
\begin{gather}
 L(x) + V^\prime (x) f(x,\theta_f) - \frac{1}{4} \bigg[ V^\prime(x) G(x,\theta_G) + L_u(x) \bigg] \nonumber \\
\cdot  R_u^{-1}(x) \bigg[ V'(x) G(x,\theta_G) + L_u(x) \bigg]^{\mathrm{T}}  + \frac{1}{4} \bigg[ V^\prime(x) K(x,\theta_K) + L_a(x) \bigg] \nonumber \\ 
\cdot  R_a^{-1}(x) \bigg[ V'(x) K(x,\theta_K) + L_a(x) \bigg]^{\mathrm{T}}=0, \quad  x \in \mathcal{S},\label{eqn: HJI}\\
 V(0)=0, \label{eqn: lyap1} \\
 V(x)>0, \quad x \in \mathcal{S}\backslash\{0\},\label{eqn: lyap2} \\
  V(x)  \rightarrow \infty \text { as } x \rightarrow \partial \mathcal{S},\label{eqn: lyap3} \\
V'(x) \bigg[ f(x,\theta_f) - \frac{1}{2} G(x,\theta_G) R_u^{-1}(x) L_u^{\rm T}(x) %\nonumber \\ 
+ \frac{1}{2} K(x,\theta_K) R_a^{-1}(x) L_a^{\rm T}(x) \nonumber \\
- \frac{1}{2} G(x,\theta_G) R_u^{-1}(x) G^{\rm T}(x,\theta_G) V'^{\rm T}(x) %\nonumber \\
+ \frac{1}{2} K(x,\theta_K) R_a^{-1}(x) K^{\rm T}(x,\theta_K) V'^{\rm T}(x) 
\bigg] \nonumber\\ 
\leq-\frac{\gamma}{T_{\mathrm{p}}}\left(\alpha V^p(x)+\beta V^q(x)\right)^{r},\quad  x \in \mathcal{S}, \label{eqn: lyap4}  
\end{gather}
where 
\begin{align*}
\gamma \triangleq \frac{\Gamma\left(\frac{1-r p}{q-p}\right) \Gamma\left(\frac{r q-1}{q-p}\right)}{\alpha^r \Gamma(r)(q-p)}\left(\frac{\alpha}{\beta}\right)^{\frac{1-r p}{q-p}}.
\end{align*}
 If  either $\mathcal{S}$ is bounded or both  $\mathcal{S}$ is unbounded and $V(\cdot)$ is coercive, then  there exist a control parameter vector $ \theta_\mathrm{c} \in \mathbb{R}^{N_\mathrm{c}}$ and an adversary parameter vector $ \theta_\mathrm{a} \in \mathbb{R}^{N_\mathrm{a}}$ such that with the feedback control strategy
\begin{align} 
u(t)= & u^\star(x(t),\theta_\mathrm{c}) \nonumber \\ 
&\triangleq - \frac{1}{2} R_u^{-1}(x(t)) \big[L_u(x(t)) + V'(x(t)) G(x(t),\theta_G) \big]^{\rm T},\quad 
 x(t) \in \mathcal{S},\quad t \geq 0,
\label{eqn_control}
\end{align}
and the feedback adversary strategy
\begin{align} 
a(t)= & a^\star(x(t),\theta_\mathrm{a}) \nonumber \\  
& \triangleq  \frac{1}{2} R_a^{-1}(x(t)) \big[L_a(x(t)) + V'(x(t)) K(x(t),\theta_K) \big]^{\rm T},\quad x(t) \in \mathcal{S},\quad t \geq 0,
\label{eqn_adversary}
\end{align}
the zero solution $x(t) \equiv 0$ to \eqref{closed-loop} is safely predefined-time stable with  predefined time $T_{\mathrm{p}}$ with respect to the set of admissible states $\mathcal{S}$. Furthermore, if $x_0 \in \mathcal{S}$, then the 
pair of feedback strategies $\left(u^{\star}(\cdot), d^{\star}(\cdot)\right)$ is the Nash equilibrium of the two-player zero-sum game in the sense of \eqref{eq:minimax}
and the Nash value is given by  \eqref{eqn_cost_hjb}.
\end{corollary}  
\begin{proof}
The proof follows immediately from Theorem \ref{Thm: general Nash equilibrium} with
$\mathcal{D} = \mathbb{R}^n$,
$U = \mathbb{R}^{m_u}$, $A = \mathbb{R}^{m_a}$,
$\theta_\mathrm{s}=[ \theta_f^\mathrm{T},\ \theta_G ^\mathrm{T},\ \theta_K^\mathrm{T}]^\mathrm{T}$,
$F(x,\theta_\mathrm{s}, u,a) = f(x,\theta_f) + G(x,\theta_G) u+K(x,\theta_K)a$, and
$r(x, u, a) =  L(x) + L_u(x) u  + L_a(x) a+ u^{\rm T} R_u(x) u-a^{\rm T} R_a(x) a$. Specifically, the Hamiltonian  function becomes
\begin{gather}
H(x,\theta_\mathrm{s},u,a,V^{\prime \mathrm{T}}(x)) \nonumber \\=L(x)+L_{u}(x) u+L_{a}(x)a+u^{\mathrm{T}} R_u(x) u -a^{\mathrm{T}} R_a(x) a  \nonumber \\
+V^\prime(x)\Big (f(x,\theta_f)+G(x,\theta_G)u+K(x,\theta_K)a\Big),\nonumber\\ 
  (x,\theta_\mathrm{s},u,a)\in \mathcal{S}\times \mathbb{R}^{N_\mathrm{s} }  \times \mathbb{R}^{m_u}\times \mathbb{R}^{m_a}. \label{proof: Hamiltonian}
\end{gather} 
Substituting $u=u^\star(x,\theta_\mathrm{c} )$ and $a=a^\star(x,\theta_\mathrm{a} )$, given by \eqref{eqn_control} and \eqref{eqn_adversary}, into the Hamiltonian function \eqref{proof: Hamiltonian}  yields \eqref{eqn: HJI}, and hence, 
\begin{align*}
    H(x,\theta_\mathrm{s},u^\star(x,\theta_\mathrm{c} ), a^\star(x,\theta_\mathrm{a} ),V^{\prime \mathrm{T}}(x))= 0, \quad  x \in \mathcal{S}, 
\end{align*}
which implies \eqref{thm H=0}.

Next, using \eqref{thm H=0}, the Hamiltonian function \eqref{proof: Hamiltonian} can be written as 
\begin{align*} 
 H(x,\theta_\mathrm{s},u,a,V^{\prime \mathrm{T}}(x))=& H(x,\theta_\mathrm{s},u,a,V^{\prime \mathrm{T}}(x)) \\
&-H(x,\theta_\mathrm{s},u^\star(x,\theta_\mathrm{c} ), a^\star(x,\theta_\mathrm{a} ),V^{\prime \mathrm{T}}(x)) \\
=& r(x, u, a)+V^{\prime}(x)\Big(f(x,\theta_f)+G(x,\theta_G) u  \\ & +K\left(x,\theta_K\right) a \Big)-r(x, u^\star(x,\theta_\mathrm{c}),a^\star(x,\theta_\mathrm{a} )) \\&
-V^{\prime}(x)\Big(f(x,\theta_f)+G(x,\theta_G) u^\star(x,\theta_\mathrm{c})\\& +K(x,\theta_K)a^\star(x,\theta_\mathrm{a})\Big) \\=&\Big( L_a(x)+V^{\prime}(x) K(x,\theta_K)\Big)\Big(a-a^\star(x,\theta_\mathrm{a})\Big)  \nonumber \\& -a^{\mathrm{T}} R_a(x) a+a^{\star\mathrm{T}}(x,\theta_\mathrm{c}) R_a(x) a^\star(x,\theta_\mathrm{a})\nonumber \\ & +\Big( L_u(x)+V^{\prime}(x) G(x,\theta_G)\Big)\Big(u-u^\star(x,\theta_\mathrm{c})\Big)  \nonumber \\&+u^{\mathrm{T}} R_u(x) u-u^{\star\mathrm{T}}(x,\theta_\mathrm{c}) R_u(x) u^\star(x,\theta_\mathrm{c}),
\end{align*}
which, using \eqref{eqn_control} and \eqref{eqn_adversary}, yields
\begin{align*}
 H(x,\theta_\mathrm{s},u,a,V^{\prime \mathrm{T}}(x)) \nonumber  
= &2 a^{\star\mathrm{T}}(x,\theta_\mathrm{a}) R_a(x)\Big(a-a^\star(x,\theta_\mathrm{a})\Big)  \nonumber \\& -a^{\mathrm{T}} R_a(x) a+a^{\star\mathrm{T}}(x,\theta_\mathrm{c}) R_a(x) a^\star(x,\theta_\mathrm{a}),\nonumber \\
&-2 u^{\star\mathrm{T}}(x,\theta_\mathrm{c}) R_u(x)\Big(u-u^\star(x,\theta_\mathrm{c})\Big) \nonumber \\
&+u^{\mathrm{T}} R_u(x) u\nonumber -u^{\star\mathrm{T}}(x,\theta_\mathrm{c}) R_u(x) u^\star(x,\theta_\mathrm{c})\\
=& -\Big(a-a^\star(x,\theta_\mathrm{a})\Big)^{\mathrm{T}} R_a(x)\Big(a-a^\star(x,\theta_\mathrm{a})\Big) \nonumber \\ &+\Big(u-u^\star(x,\theta_\mathrm{c})\Big)^{\mathrm{T}} R_u(x)\Big(u-u^\star(x,\theta_\mathrm{c})\Big).
\end{align*}
Now, since $R_a(x)\succ 0,\ x \in  \mathcal{S},$ and $R_u(x)\succ 0,\ x \in  \mathcal{S},$ it follows that 
\begin{align*}
H(x,\theta_\mathrm{s}, u^\star(x,\theta_\mathrm{c}), a,V^{\prime \mathrm{T}}(x)) &=-\Big(a-a^\star(x,\theta_\mathrm{a})\Big)^{\mathrm{T}} R_a(x)\Big(a-a^\star(x,\theta_\mathrm{a})\Big) \\
& \leq 0,\quad (x,a)\in \mathcal{S}  \times \mathbb{R}^{m_a},
\end{align*}
and
\begin{align*}
H(x,\theta_\mathrm{s},u, a^\star(x,\theta_\mathrm{a}),V^{\prime \mathrm{T}}(x)) &=\Big(u-u^\star(x,\theta_\mathrm{c}\Big)^{\mathrm{T}} R_u(x)\Big(u-u^\star(x,\theta_\mathrm{c})\Big) \\
& \geq 0,\quad (x,u)\in \mathcal{S}  \times \mathbb{R}^{m_u},
\end{align*}
which imply \eqref{thm H<0} and \eqref{thm H>0}.

Next, since $x=0$ is a local minimum of the continuously differentiable $V(\cdot)$, we have $ V^\prime(0) = 0$. Hence, it follows from \eqref{eqn_control} and \eqref{eqn_adversary} that $u^\star(0,\theta_\mathrm{c} ) = 0$ and $a^\star(0,\theta_\mathrm{a} ) = 0$, which verify \eqref{thm ustar} and \eqref{thm astar}. Finally, \eqref{eqn: lyap1}--\eqref{eqn: lyap4} are equivalent to
\eqref{eq:them 2 lyap1}--\eqref{eq:them 2 lyap4} with $u^\star(x,\theta_\mathrm{c})$ and $a^\star(x,\theta_\mathrm{a})$ given by \eqref{eqn_control} and \eqref{eqn_adversary}. Now, the result is an immediate consequence of Theorem \ref{Thm: general Nash equilibrium}.\frQED
\end{proof}

\begin{remark}
Since the Hamiltonian function \eqref{proof: Hamiltonian} is convex in $u$ and concave in $a$ for every $x \in \mathcal{S}$, we derive the saddle point feedback strategies \eqref{eqn_control} and \eqref{eqn_adversary} by setting $\frac{\partial H}{\partial u}=0$ and $\frac{\partial H}{\partial a}=0$.  \frqed
\end{remark}

The robust optimal safe predefined time stabilization problem is equivalent to solving the steady-state  HJI equation \eqref{eqn: HJI}, which is generally difficult to solve, except for special cases. To circumvent the complexity in solving the steady-state HJI equation, we consider a \textit{robust inverse optimal feedback control problem} \cite{freeman1996inverse,l2017differential,l2017differentialCTA}, wherein we
parameterize a class of safely predefined time stabilizing pairs of feedback strategies that constitute a saddle point of a \textit{derived} cost functional instead of seeking a \textit{given} cost functional  saddle point, thus providing flexibility in specifying the adversarially robust optimal controller. 

\begin{corollary}\label{theorem_inverse} 
Consider the parameter-dependent nonlinear affine dynamical
system \eqref{eq: affine system} with cost functional \eqref{eqn_generic_performance_measure}. Let the controller $u(\cdot)$ be a player that \textit{minimizes} the cost functional \eqref{eqn_generic_performance_measure} and let the adversary $a(\cdot)$ be a player that \textit{maximizes} the cost functional \eqref{eqn_generic_performance_measure}. Let $\mathcal{S} \subset  \mathbb{R}^{n}$ be a set of admissible states with $0 \in \mathcal{S}$ and let $T_{\mathrm{p}} >0$ be a predefined time. For every $ x_0 \in \mathcal{S}$, let $\mathcal{F}\left( x_0,\mathcal{S},T_{\mathrm{p}} \right)\subset \mathcal{U} \times \mathcal{D}$ be the set of safely predefined time stabilizing pairs of feedback strategies and suppose that $\mathcal{F}\left( x_0,\mathcal{S},T_{\mathrm{p}}\right)$ is nonempty.
Assume that there exist a continuously differentiable function $V : \mathcal{S} \to \mathbb{R},$ continuous functions $L_{u}: \mathcal{S} \rightarrow \mathbb{R}^{1 \times m_u}$ and $L_{a}: \mathcal{S} \rightarrow \mathbb{R}^{1 \times m_a}$, continuous positive-definite matrix functions $R_u: \mathcal{S} \rightarrow \mathbb{R}^{m_u \times m_u}$ and $R_a: \mathcal{S} \rightarrow \mathbb{R}^{m_a \times m_a}$, system parameter vectors $\theta_f\in \mathbb{R}^{N_f},\ \theta_G \in \mathbb{R}^{N_G },$ and $\theta_K \in \mathbb{R}^{N_K}$, and real numbers $\alpha,\ \beta,\ p,\ q,\ r>0$ such that $p r<1,\ q r>1,$ and \eqref{eqn: lyap1}-\eqref{eqn: lyap4} hold.
If  either $\mathcal{S}$ is bounded or both  $\mathcal{S}$ is unbounded and $V(\cdot)$ is coercive,  then with the feedback control strategy \eqref{eqn_control} and the feedback adversary strategy \eqref{eqn_adversary},
the zero solution $x(t) \equiv 0$ to \eqref{closed-loop}
is safely predefined-time stable with  predefined time $T_{\mathrm{p}}$ with respect to the set of admissible states $\mathcal{S}$, the cost functional \eqref{eqn_generic_performance_measure},
with
\begin{align} \label{eqn_optimal_control_affine_2_inverse}
L(x) = &u^{\star\rm T}(x,\theta_\mathrm{c}) R_u(x) u^\star(x,\theta_\mathrm{c})- V^\prime(x) f(x,\theta_f) 
-d^{\star\rm T}(x,\theta_\mathrm{a}) R_a(x) a^\star(x,\theta_\mathrm{a}),\quad  x \in \mathcal{S},
\end{align} 
has a saddle point in the sense of \eqref{eq:minimax}, and the Nash value is given by \eqref{eqn_cost_hjb}.
\end{corollary}
\begin{proof}
The proof is similar to the proof of Corollary \ref{Thm: affine Nash equilibrium} and hence is omitted.\frQED
\end{proof}

\begin{remark}
The following observations are important. 

1) The function $L(x)$ in the running cost \eqref{eq: performance integrand} 
 given by \eqref{eqn_optimal_control_affine_2_inverse} explicitly depends on the nonlinear system dynamics, the safely predefined time stabilizing Lyapunov function of the closed-loop system, and the feedback Nash equilibrium of the two-player zero-sum game. Note that the coupling is introduced via the steady-state HJI equation \eqref{eqn: HJI}. 

2) In light of \eqref{eqn_control} and \eqref{eqn_adversary}, the vector functions $L_{u}(x)$ and $L_{a}(x)$ in the running cost \eqref{eq: performance  integrand}  provide flexibility in the controller and adversary synthesis since $L_{u}(x)$ and $L_{a}(x)$ are arbitrary functions of $x \in \mathcal{S}$ subject to condition \eqref{eqn: lyap4}, a stability condition in Theorem \ref{Thm: affine Nash equilibrium}. \frqed 
\end{remark}
In light of the above observations, note that by varying the parameters in the barrier Lyapunov function and the running cost, we can characterize a family of safely predefined-time stabilizing pairs of feedback strategies that can satisfy closed-loop system response requirements.

\section{PINNs for Robust Optimal Safe Control}\label{sec:PINN}

Unlike the robust \textit{inverse} optimal safe stabilization problem, for which the Nash value is known, the \textit{robust optimal safe predefined-time stabilization} problem is equivalent to solving the steady-state HJI equation \eqref{eqn: HJI} subject to the constraints \eqref{eqn: lyap1}--\eqref{eqn: lyap4}, which is generally difficult to solve, except for special cases. 

%In this section, b
Building on the results of \cite{lu2021physics}, we design a \textit{physics-informed learning} framework to approximate the safely predefined-time stabilizing solution $V(\cdot)$ of the steady-state HJI equation \eqref{eqn: HJI}. Specifically, invoking the universal approximation property of deep
neural networks \cite{hornik1990universal}, we introduce a surrogate model $\hat{V}(\cdot,w)$ with $w\in\mathbb{R}^N$ representing the trainable model parameters. To ensure the satisfaction of the constraints \eqref{eqn: lyap1}--\eqref{eqn: lyap3}, let 
\begin{equation}\label{eq:nn_vhat}
    \hat{V}(x,w) = h({V_{\text{NN}}(x,w)})B(x), \quad (x,w) \in \mathcal{S} \times \mathbb{R}^N,
\end{equation}
where $h:\mathbb{R} \to (0, \infty)$ is a user-defined continuously differentiable function, $V_{\text{NN}}:\mathcal{S}\times \mathbb{R}^N\to\mathbb{R}$ is a standard fully-connected neural network, and $B:\mathcal{S}\to \mathbb{R}$ is a user-defined continuously differentiable function satisfying $B(0)=0$, $B(x)>0, \ x \in \mathcal{S} \backslash\{0\}$, and $B(x) \rightarrow \infty$ as $x \rightarrow \partial \mathcal{S}$. If $\mathcal{S}$ is unbounded, then $B(\cdot)$ is additionally coercive. The construction of $h(\cdot)$ and $B(\cdot)$ for each numerical example is provided in Section~\ref{sec:Simulation Results}.

To learn an approximation $\hat{V}(\cdot,w)$ of $V(\cdot)$, we randomly sample $M$ points in $\mathcal{S}$ and let $S_{\text{col}}  \triangleq \{x_1, \dots, x_M\} \subset \mathcal{S}$ denote the set of \textit{collocation points}. The learning procedure is formulated as a \textit{constrained} optimization problem, where, at the collocation points, the loss function $\mathcal{E}(\cdot)$ penalizes
violations of the steady-state HJI equation \eqref{eqn: HJI} and the constraint function $l(\cdot,\cdot)$ enforces the differential inequality \eqref{eqn: lyap4} to guarantee safe predefined-time stability. Specifically, the constrained optimization problem is of the form
\begin{gather}
    \min_{w \in \mathbb{R}^N} \mathcal{E} (w) \nonumber \\ 
    \text{subject to}\ l(x,w) \leq 0, \quad x \in S_{\text{col}},\label{contrained_opt}
\end{gather}
where the loss function $\mathcal{E}(\cdot)$ and the constraint function $l(\cdot,\cdot)$ are defined by 
\begin{align} \label{object}
\mathcal{E}(w) \triangleq &\sum_{x \in S_{\text{col}}} \Big| L(x) + \hat{V}_x(x,w)f(x,\theta_f) \nonumber \\ 
                          & - \frac{1}{4} \left( \hat{V}_x(x,w)G(x,\theta_G) + L_u(x) \right) 
                           R^{-1}_u(x) \left( \hat{V}_x(x,w)G(x,\theta_G) +L_u(x)  \right)^{\mathrm{T}} \nonumber \\ 
                          & + \frac{1}{4} \left( \hat{V}_x(x,w)K(x,\theta_K) + L_a(x) \right)
                           R^{-1}_a(x) \left( \hat{V}_x(x,w)K(x,\theta_K) +L_a(x)  \right)^{\mathrm{T}} \Big|^2, \nonumber \\ & \hspace{10.1 cm}
w \in \mathbb{R}^N,
\end{align} 
 and 
\begin{align} \label{constraint}
    l(x,w) \triangleq & ~ \hat{V}_x(x,w)\Big( f(x,\theta_f)  -\frac{1}{2} G(x,\theta_G)R^{-1}_u(x)L_u^T(x) \nonumber \\ 
               & +\frac{1}{2} K(x,\theta_K)R^{-1}_a(x)L_a^T(x) \nonumber\\ 
               & -\frac{1}{2} G(x,\theta_G)R^{-1}_u(x)G^T(x,\theta_G)\hat{V}^T_x(x,w) \nonumber \\ 
               & + \frac{1}{2} K(x,\theta_K)R^{-1}_a(x)K^T(x,\theta_K)\hat{V}^T_x(x,w) \Big) \nonumber \\ 
               & + \frac{\gamma}{T_p} \Big(\alpha\hat{V}^p(x,w) + \beta\hat{V}^q(x,w) \Big)^{r}, %\nonumber \\ & \hspace{3.2cm }
             \quad (x,w) \in \mathcal{S}_{\text{col}} \times \mathbb{R}^N.
    \end{align}
    The constrained optimization problem \eqref{contrained_opt} can be numerically addressed using the \textit{augmented Lagrangian method} (see, for example,~\cite{nocedal1999numerical} and \cite{bertsekas2014constrained}). 
In particular, the constrained optimization problem~\eqref{contrained_opt} is converted into a sequence of unconstrained optimization subproblems through an iterative procedure. 
In each iteration, the objective function of the unconstrained optimization subproblem is the sum of the loss function $\mathcal{E}(\cdot)$ and a penalty term for the constraints. 
Hence, the optimization subproblem at iteration $k$ takes the form
\begin{equation}\label{subprob_itek}
    \min_{w \in \mathbb{R}^N} \mathcal{E}_k(w), 
\end{equation}
where $\mathcal{E}_k(\cdot)$ is the loss function at iteration $k$ given by
\begin{align} \label{sub_itek}
\mathcal{E}_k(w)  \triangleq \mathcal{E}(w) + \sum_{x \in S_{\text{col}}} \Big(&\mu_{k-1} \mathbbm{1}_{\{l(x,w) \geq 0 \lor \lambda_{k-1}(x) >0\}} l^2(x,w) %\nonumber  \\
       + %&
       \lambda_{k-1}(x) l(x,w)\Big),
\end{align}
with $\lor$ representing the or operator.

For every iteration $k \in \mathbb{Z}_{+}$, the Lagrangian multipliers $\mu_k$ and $\lambda_k(\cdot)$ in \eqref{sub_itek} are updated as 
\begin{align}
    \mu_k &= \delta \mu_{k-1}, \label{update_multiplier}\\ 
    \lambda_k(x) &= \max\{0,\lambda_{k-1}(x) + 2 \mu_{k-1}l(x,w) \} ,\quad x \in S_{\text{col}}, \label{update_multiplier2}
\end{align}
with $\delta,\ \mu_0 \in \mathbb{R}$ being tunable hyperparameters, and $\lambda_0(x)$ is initialized to be $0$ for every $x \in S_{\text{col}}$. The optimization subproblem~\eqref{subprob_itek} can be solved using modern gradient-based or (quasi-) Newton-based numerical solvers, such as Adam~\cite{kingma2014adam} or L-BFGS~\cite{zhu1997algorithm}. 

Let $w_k$ denote the minimizer of the optimization subproblem~\eqref{subprob_itek} at iteration $k$. Substituting $w_k$ into \eqref{eq:nn_vhat} yields $\hat{V}(\cdot,w_k)$, an approximation of the value function. 
Furthermore, substituting $\hat{V}(\cdot, w_k)$ into~\eqref{eqn_control} and~\eqref{eqn_adversary} yields the approximate Nash feedback control and adversary strategies $\hat{u}(\cdot,w_k)$ and $\hat{a}(\cdot,w_k)$, respectively. Algorithm~\ref{alg:hpinn} presents the pseudocode for the proposed adversarial physics-informed learning framework for robust optimal safe control, executed for a given number of iterations $\mathcal{K}$. The architecture of Algorithm~\ref{alg:hpinn} is shown in Fig. ~\ref{fig:sketch_pinn}.  
%Moreover, in Theorem~\ref{th_converg}, we present the asymptotic convergence of Algorithm~\ref{alg:hpinn}. 
%\begin{theorem}\label{th_converg}
%    Let $V:\mathcal{S}\to \mathbb{R}$ be the unique stabilizing solution of the steady-state HJI equation~\eqref{eqn: HJI}, which satisfies the constraints \eqref{eqn: lyap1}--\eqref{eqn: lyap4}. Let $\hat{V}(\cdot, w_K)$ denote the output of Algorithm~\ref{alg:hpinn}. Upon the well training of the neural network (i.e., the optimization problem~\eqref{subprob_itek} is well solved at each step), $\hat{V}(\cdot,w_K)$ converges to $\hat{V}^*(\cdot)$, an approximation of $V(\cdot)$, as $K\to \infty$.
%\end{theorem}
%\begin{proof}
%Once the neural network is well trained, the subproblem~\eqref{subprob_itek} is effectively solved at each iteration of Algorithm~\ref{alg:hpinn}. Consequently, by applying standard convergence results from the augmented Lagrangian method (which we omit here; see, for example, \cite{nocedal1999numerical,zhu1997algorithm}), the solution $\hat{V}(\cdot,w_K)$, of the unconstrained optimization problem converges to $\hat{V}^*(\cdot)$, as $K \to \infty$, which corresponds to the solution of the original constrained optimization problem~\eqref{contrained_opt}.
%\end{proof}
%Let us assume that the value and the optimizer of the subproblem~\eqref{sub_itek} converges to the original 
%constrained optimization problem~\eqref{contrained_opt} after $K$ iterations. 
\begin{algorithm}
\caption{Training procedure of Adversarial PINN}
\hspace*{\algorithmicindent} \textbf{Hyperparameters}: $\alpha, \beta, \gamma, T_{\mathrm{p}}, p, q, r, \delta, \mu_0$\\
\hspace*{\algorithmicindent} \textbf{Input}: Collocation points $S_{\text{col}}\triangleq\{x_1, \dots, x_M \}\subset \mathcal{S}$\\
   \hspace*{\algorithmicindent} \textbf{Output}: Nash value $\hat{V}(\cdot, w_\mathcal{K})$ and Nash equilibrium $\left(\hat{u}(\cdot,w_\mathcal{K}),\hat{a}(\cdot,w_\mathcal{K})\right)$
\begin{algorithmic}[1]
\Procedure{}{}
   \State $\lambda_0 (x) \gets 0$ for every $x \in S_{\text{col}} $
   \State Initialize network parameters $w_0\in\mathbb{R}^N$
   \For{$k = 1 \cdots \mathcal{K}$}
   \State $\Tilde{\mathcal{E}} \gets \mathcal{E}(w_{k-1})$
   \Comment{\eqref{object}}
   \State $\Tilde{l}(x) \gets l(x,w_{k-1})$  
   \Comment{\eqref{constraint}}
 \State $\mathcal{E}_k \gets \Tilde{\mathcal{E}} + \sum_{x \in S_{\text{col}}} \left(\mu_{k-1} \, \mathbbm{1}_{\{\Tilde{l}(x) \geq 0 \lor \lambda_{k-1}(x) > 0\}} \, \Tilde{l}^2(x) + \lambda_{k-1}(x) \, \Tilde{l}(x)\right)
$
   \State $w_{k} \gets \argmin_w \mathcal{E}_k$
   \State $\mu_k \gets \delta \mu_{k-1}$
    \State $ \begin{aligned}[t] & \lambda_k(x) \gets \max\{0, \lambda_{k-1}(x) + 2\mu_{k-1} \Tilde{l}(x)\}\\ & \text{for every} ~ x \in S_{\text{col}} \end{aligned}$
   \EndFor
\EndProcedure
\end{algorithmic}
\label{alg:hpinn}
 \end{algorithm}
 \begin{figure}
     \centering
     \includegraphics[width=0.5\linewidth]{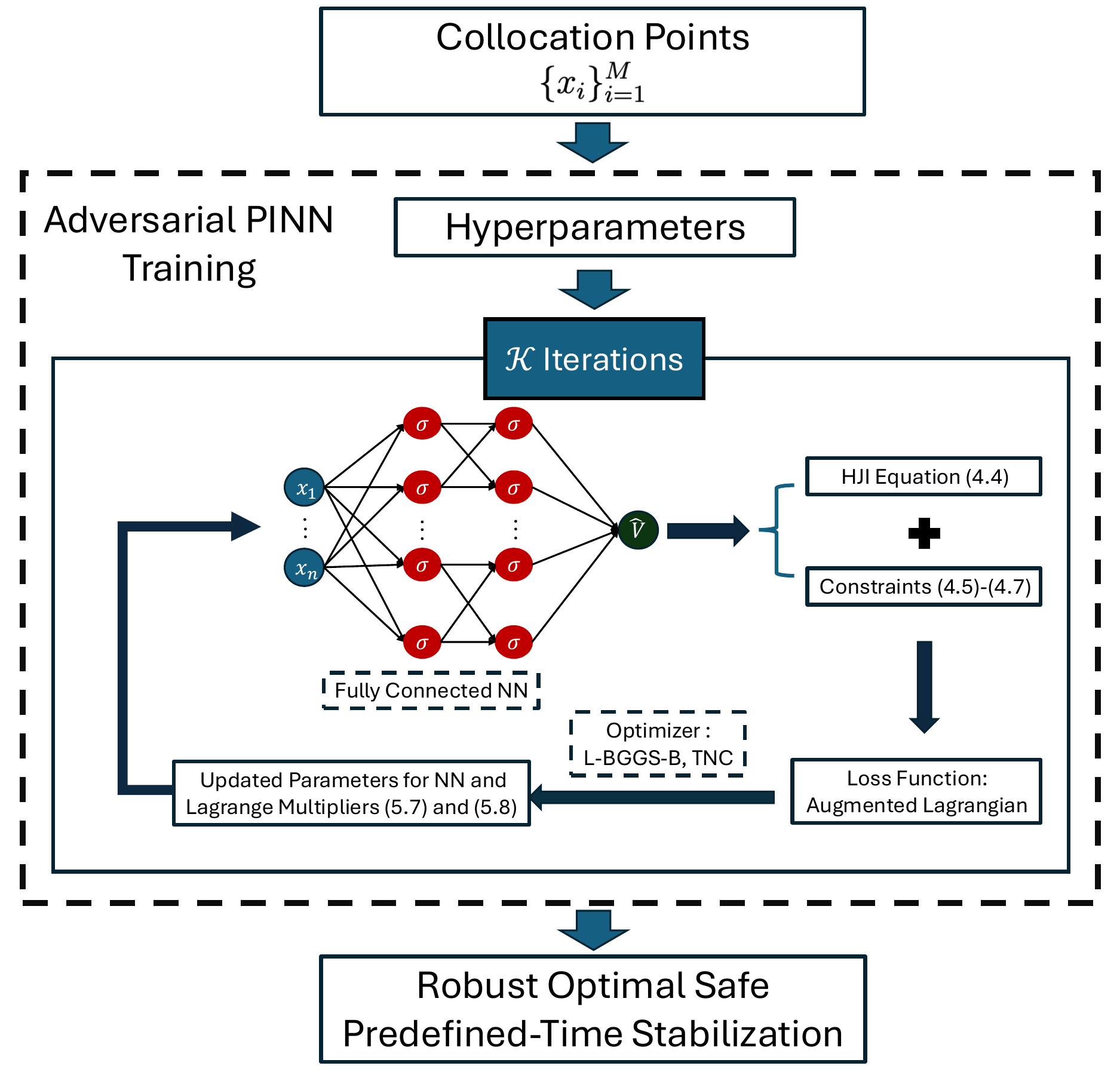}
     \caption{\small Adversarial physics-informed learning architecture to solve the robust optimal safe predefined-time stabilization problem. A set of collocation points $\left\{x_i\right\}_{i=1}^M$ is randomly sampled in $\mathcal{S}$. The training is formulated as a constrained optimization problem: at each collocation point, the loss function penalizes deviations from the HJI equation~\eqref{eqn: HJI}, while the constraint function enforces the differential inequality \eqref{eqn: lyap4} to ensure safe predefined-time stability. Constraints \eqref{eqn: lyap1}-\eqref{eqn: lyap3} are satisfied by the design of the PINN architecture. This constrained optimization problem is solved iteratively using the augmented Lagrangian method, with the Lagrange multipliers updated at each iteration according to \eqref{update_multiplier} and \eqref{update_multiplier2}. }
     \label{fig:sketch_pinn}
 \end{figure}

\begin{remark}
 Note that the steady-state HJI equation \eqref{eqn: HJI}  may have multiple solutions. However, the Nash value $V(\cdot)$ is the unique safely predefined-time stabilizing solution. Hence, when using PINNs to solve the robust optimal safe predefined-time stabilization problem, enforcing the stability conditions \eqref{eqn: lyap1}-\eqref{eqn: lyap4} is imperative. Conditions \eqref{eqn: lyap1}-\eqref{eqn: lyap3} are satisfied by the design of the PINN architecture, while \eqref{eqn: lyap4} is imposed as a constraint in the learning procedure. \frqed
\end{remark}

\section{Illustrative Numerical Examples} \label{sec:Simulation Results}
In this section, we present two numerical examples to illustrate the developed robust optimal safe predefined time control and physics-informed learning framework. 

Consider the nonlinear affine dynamical system given by
\begin{align}
  \dot{x}_1(t)=& - \min\{0, x_1(t)\}+u_1(t)+a_1(t),\quad  x_1(0)=x_{10},\quad  t \geq 0, \label{eq:lor 1}\\
 \dot{x}_2(t)=&- \max \{0, x_2(t)\}
 +u_2(t)+a_2(t),\quad x_2(0)=x_{20}. \label{eq:lor 2}
\end{align} Note that \eqref{eq:lor 1} and \eqref{eq:lor 2} can be cast in the form of  \eqref{eq: affine system} with $n=2,\ m_u=2, \ m_a=2,\ x=\left[x_{1},\ x_{2}\right]^{\mathrm{T}},\ u=\left[u_{1},\ u_{2}\right]^{\mathrm{T}},\ a=\left[a_{1},\ a_{2}\right]^{\mathrm{T}},$  
$
f(x,\theta_f)=-[
 \min \{0, x_1\},\
 \max \{0, $ $x_2\}
]^{\mathrm{T}},\
$
$ G(x,\theta_G)=I_{2},$ 
and
$
K(x,\theta_K)=I_{2}.
$

Let $s: \mathbb{R}^2 \rightarrow \mathbb{R}$ be a continuous function such that $s(0)>0$. Define the set of admissible states $\mathcal{S}$ as a zero-strict superlevel set of $s(\cdot)$ given by   
\begin{equation*} 
\mathcal{S} = \left\{x \in \mathbb{R}^2: s(x) >0 \right\}.
\end{equation*}
Note that $\mathcal{S}$ is an open set with $0 \in \mathcal{S}$.

Next, given the set of admissible states $\mathcal{S}$ and a predefined time $T_{\mathrm{p}}$, we use Corollary \ref{theorem_inverse} to synthesize an inverse safely predefined time stabilizing Nash equilibrium $\left(u^{\star}(\cdot), a^{\star}(\cdot)\right)$. 
In particular, we address the case of both a bounded and an unbounded set of admissible states.

\subsection{Bounded Safe Set}
Let $s(x)=\min\{s_1(x), s_2(x)\},\ x \in \mathbb{R}^2,$ where  $s_1(x) \triangleq 1-x_1^2,\ x \in \mathbb{R}^2,$ and $s_2(x) \triangleq 1-x_2^2,\ x \in \mathbb{R}^2$. Furthermore, let $V\left(x\right)=\frac{x_1^2}{2s_1(x)}+\frac{x_2^2}{2s_2(x)}, \ x \in \mathcal{S},$ be the Nash value. The terms of the running cost \eqref{eq: performance integrand} are given by
\begin{align*}
L(x)=& \frac{1}{2}\left( \frac{\left\lceil x_1\right\rfloor^{\gamma_1}}{s_1^{\frac{\gamma_1-1}{2}}(x)} + \frac{\left\lceil x_1\right\rfloor^{\gamma_2}}{s_1^{\frac{\gamma_2-1}{2}}(x)} \right)^2 %\\ & 
+\frac{1}{2}\left( \frac{\left\lceil x_2\right\rfloor^{\gamma_1}}{s_2^{\frac{\gamma_1-1}{2}}(x)} + \frac{\left\lceil x_2\right\rfloor^{\gamma_2}}{s_2^{\frac{\gamma_2-1}{2}}(x)} \right)^2 \\ &+\frac{x_1}{s_1(x)}\left(1+ \frac{x_{1}^{2}}{s_1(x)} \right) \min \{0, x_1\}%\\ &
+\frac{x_2}{s_2(x)}\left( 1+ \frac{x_{2}^{2}}{s_2(x)}\right) \max \{0, x_2\},\quad x \in \mathcal{S},\\
L_{u}(x)=&L_{a}(x) = %\\ =&
\begin{bmatrix}
\displaystyle \frac{\left\lceil x_1\right\rfloor^{\gamma_1}}{s_1^{\frac{\gamma_1-1}{2}}(x)} + \frac{\left\lceil x_1\right\rfloor^{\gamma_2}}{s_1^{\frac{\gamma_2-1}{2}}(x)} \\[2em]
\displaystyle \frac{\left\lceil x_2\right\rfloor^{\gamma_1}}{s_2^{\frac{\gamma_1-1}{2}}(x)} + \frac{\left\lceil x_2\right\rfloor^{\gamma_2}}{s_2^{\frac{\gamma_2-1}{2}}(x)}
\end{bmatrix}^{\mathrm{T}} \\ &  \\ & \hspace{1.1cm} - \left[\frac{x_1}{s_1(x)}\left( 1+ \frac{x_{1}^{2}}{s_1(x)}\right),\ \frac{x_2}{s_2(x)}\left( 1+ \frac{x_{2}^{2}}{s_2(x)}\right)\right], %\\ & \hspace{6.2cm} 
\quad x \in \mathcal{S},
\end{align*} 
$R_u(x)=\frac{1}{4}I_{2},\ x \in \mathcal{S}$, and $R_a(x)=\frac{1}{2}I_{2},\ x \in \mathcal{S}$, where $\gamma_1 \in (0,1)$ and $\gamma_2>1$. Hence, 
with $\theta_\mathrm{c}=\left[\gamma_1,\ \gamma_2 \right]^{\mathrm{T}}$ and $\theta_\mathrm{a}=\left[\gamma_1,\ \gamma_2 \right]^{\mathrm{T}}$, the inverse Nash equilibrium is given by
\begin{equation} \label{ex bounded: Nash control}
u^\star(x,\theta_\mathrm{c}) = -2 \begin{bmatrix}
\displaystyle \frac{\left\lceil x_1\right\rfloor^{\gamma_1}}{s_1^{\frac{\gamma_1-1}{2}}(x)} + \frac{\left\lceil x_1\right\rfloor^{\gamma_2}}{s_1^{\frac{\gamma_2-1}{2}}(x)} \\[2em]
\displaystyle \frac{\left\lceil x_2\right\rfloor^{\gamma_1}}{s_2^{\frac{\gamma_1-1}{2}}(x)} + \frac{\left\lceil x_2\right\rfloor^{\gamma_2}}{s_2^{\frac{\gamma_2-1}{2}}(x)}
\end{bmatrix}, \quad x \in \mathcal{S},
\end{equation}
and
\begin{equation} \label{ex bounded: Nash adversary}
 a^\star(x,\theta_\mathrm{a}) =  \begin{bmatrix}
\displaystyle \frac{\left\lceil x_1\right\rfloor^{\gamma_1}}{s_1^{\frac{\gamma_1-1}{2}}(x)} + \frac{\left\lceil x_1\right\rfloor^{\gamma_2}}{s_1^{\frac{\gamma_2-1}{2}}(x)} \\[2em]
\displaystyle \frac{\left\lceil x_2\right\rfloor^{\gamma_1}}{s_2^{\frac{\gamma_1-1}{2}}(x)} + \frac{\left\lceil x_2\right\rfloor^{\gamma_2}}{s_2^{\frac{\gamma_2-1}{2}}(x)}
\end{bmatrix}, \quad x \in \mathcal{S}.
\end{equation}
Next, computing the time derivative of $V(x)$ along the trajectories of \eqref{eq:lor 1} and \eqref{eq:lor 2} with $u=u^\star(x,\theta_\mathrm{c})$ and $a=a^\star(x,\theta_\mathrm{a})$ given by \eqref{ex bounded: Nash control} and \eqref{ex bounded: Nash adversary}, we obtain
\begin{align*}
\dot{V}(x)= &-\left( \frac{|x_{1}|^{\gamma_1+1}}{s_1^{\frac{\gamma_1+1}{2}}(x)}  +   \frac{|x_{1}|^{\gamma_2+1}}{s_1^{\frac{\gamma_2+1}{2}}(x)}\right)\nonumber %\\  & 
-\left( \frac{|x_{2}|^{\gamma_1+1}}{s_2^{\frac{\gamma_1+1}{2}}(x)}  +   \frac{|x_{2}|^{\gamma_2+1}}{s_2^{\frac{\gamma_2+1}{2}}(x)}\right)\nonumber \\  & 
-\frac{x_1}{s_1(x)} \min \{0, x_1\}-\frac{x_2}{s_2(x)} \max \{0, x_2\}                       
\nonumber \\  & -\frac{x_1^2}{ s_1^2(x)} \left( \min \{0, x_1\} x_1 
+\frac{|x_{1}|^{\gamma_1+1}}{s_1^{\frac{\gamma_1-1}{2}}(x)}  +   \frac{|x_{1}|^{\gamma_2+1}}{s_1^{\frac{\gamma_2-1}{2}}(x)}\right) \nonumber \\  & -\frac{x_2^2}{ s_2^2(x)} \left( \max \{0, x_2\} x_2 
+\frac{|x_{2}|^{\gamma_1+1}}{s_2^{\frac{\gamma_1-1}{2}}(x)}  +   \frac{|x_{2}|^{\gamma_2+1}}{s_2^{\frac{\gamma_2-1}{2}}(x)}\right) \nonumber
\\ \leq &-\left( \frac{|x_{1}|^{\gamma_1+1}}{s_1^{\frac{\gamma_1+1}{2}}(x)} +
\frac{|x_{2}|^{\gamma_1+1}}{s_2^{\frac{\gamma_1+1}{2}}(x)} \right)\nonumber %\\  & 
-\left(\frac{|x_{1}|^{\gamma_2+1}}{s_1^{\frac{\gamma_2+1}{2}}(x)}+\frac{|x_{2}|^{\gamma_2+1}}{s_2^{\frac{\gamma_2+1}{2}}(x)}\right),\quad x \in \mathcal{S}.
\end{align*}
Now, using $|y+\bar{y}|^{\delta_1} \leq|y|^{\delta_1}+|\bar{y}|^{\delta_1}$
for every $y,\ \bar{y} \in \mathbb{R}$ and $\delta_1 \in(0,1)$,\\  and $|y+\bar{y}|^{\delta_{2}} \leq 2^{\delta_{2}-1}\left(|y|^{\delta_{2}}+|\bar{y}|^{\delta_{2}}\right)$
 for every $y,\ \bar{y} \in \mathbb{R}$ and $\delta_2>1$ \cite{mitrinovic1970analytic}, and letting $y=\frac{x_1^2}{2s_1(x)},\ x \in \mathcal{S},\ \delta_1={\frac{\gamma_1+1}{2}}, \ \bar{y}=\frac{x_2^2}{2s_2(x)},\ x \in \mathcal{S}$, and $\delta_2={\frac{\gamma_2+1}{2}}$, it follows that
\begin{align*} 
\dot{V}(x) \leq  - 2^{\frac{\gamma_1+1}{2}} V^{\frac{\gamma_1+1}{2}}(x)-2 V^{\frac{\gamma_2+1}{2}}(x), \quad x \in \mathcal{S},
\end{align*}
which implies that \eqref{eqn: lyap4} holds with $\gamma=T_{\mathrm{p}}$, $\alpha=2^{\frac{\gamma_1+1}{2}} ,\ \beta=2,\ p={\frac{\gamma_1+1}{2}},\ q=\frac{\gamma_2+1}{2}$, and $r=1$. Hence, by Corollary \ref{theorem_inverse}, the zero solution $x(t) \equiv 0$ of the closed-loop system is safely predefined-time stable.

Let $T_p = 3.4295$ so that $\gamma_1 = 0.5$ and $\gamma_2=1.5$.  
We use Algorithm~\ref{alg:hpinn} to learn the Nash value $V(\cdot)$, the Nash feedback control strategy $u^\star(\cdot,\theta_\mathrm{c})$, and the Nash feedback adversary strategy $a^\star(\cdot,\theta_\mathrm{a})$. Specifically, for the PINN surrogate $\hat{V}(x,w)$ defined by \eqref{eq:nn_vhat}, we set $h(x) = e^x, \ x\in \mathbb{R}$, and $B(x) = \frac{x_1^2}{1 - |x_1|} + \frac{x_2^2}{1 - |x_2|}, \ x \in \mathcal{S}$. 
The neural network $V_{\text{NN}}(\cdot,\cdot)$ is implemented as a fully connected neural network with six hidden layers, each comprising $100$ neurons, and employs the hyperbolic tangent activation function $\tanh(\cdot)$.  The number of collocation points is set to $M=10^4$. The hyperparameters associated with the augmented Lagrangian method are selected as $\mu_0 = 10^{-4}$ and $\delta = 2$. At each iteration of Algorithm \ref{alg:hpinn}, the optimization subproblem \eqref{subprob_itek} is solved using the L-BFGS optimizer. 

\begin{figure}[!ht] 
\centering\includegraphics[width=0.7\textwidth]{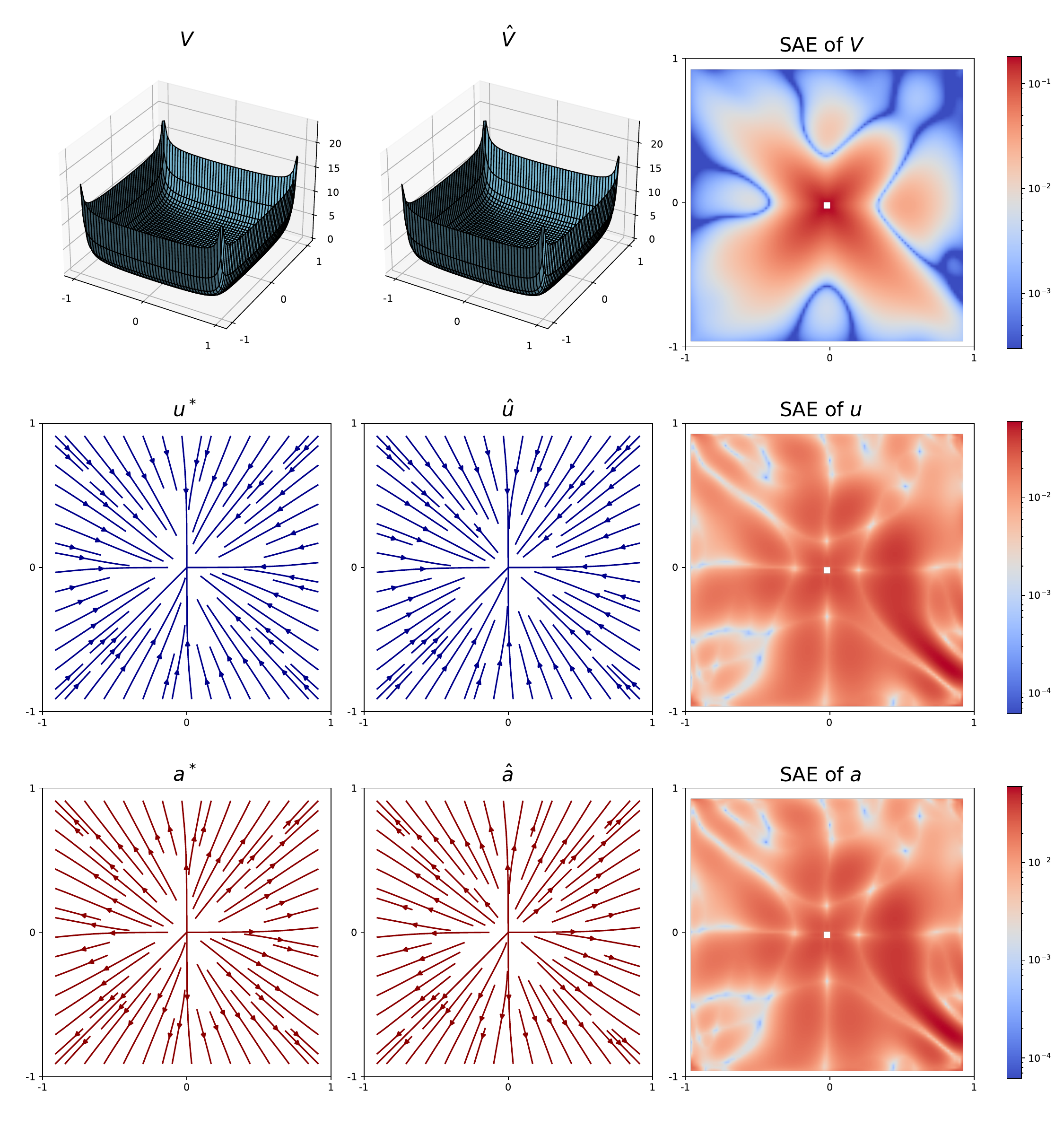}
\caption{\small Nash value, Nash feedback control strategy, and Nash feedback adversary strategy for the bounded safe set. 
(\textbf{Left}) Exact Nash value $V$, exact Nash control strategy $u^\star$, and exact Nash adversary strategy $a^\star$. 
(\textbf{Middle}) Approximate Nash value $\hat{V}$, approximate Nash control strategy $\hat{u}$, and approximate Nash adversary strategy $\hat{a}$. 
(\textbf{Right}) Symmetric absolute error for learning the Nash value, the Nash control strategy, and the Nash adversary strategy.}

\label{figure:bounded pinn appro}
\end{figure}
\begin{figure}[!ht]
\centering\includegraphics[width=0.7\textwidth]{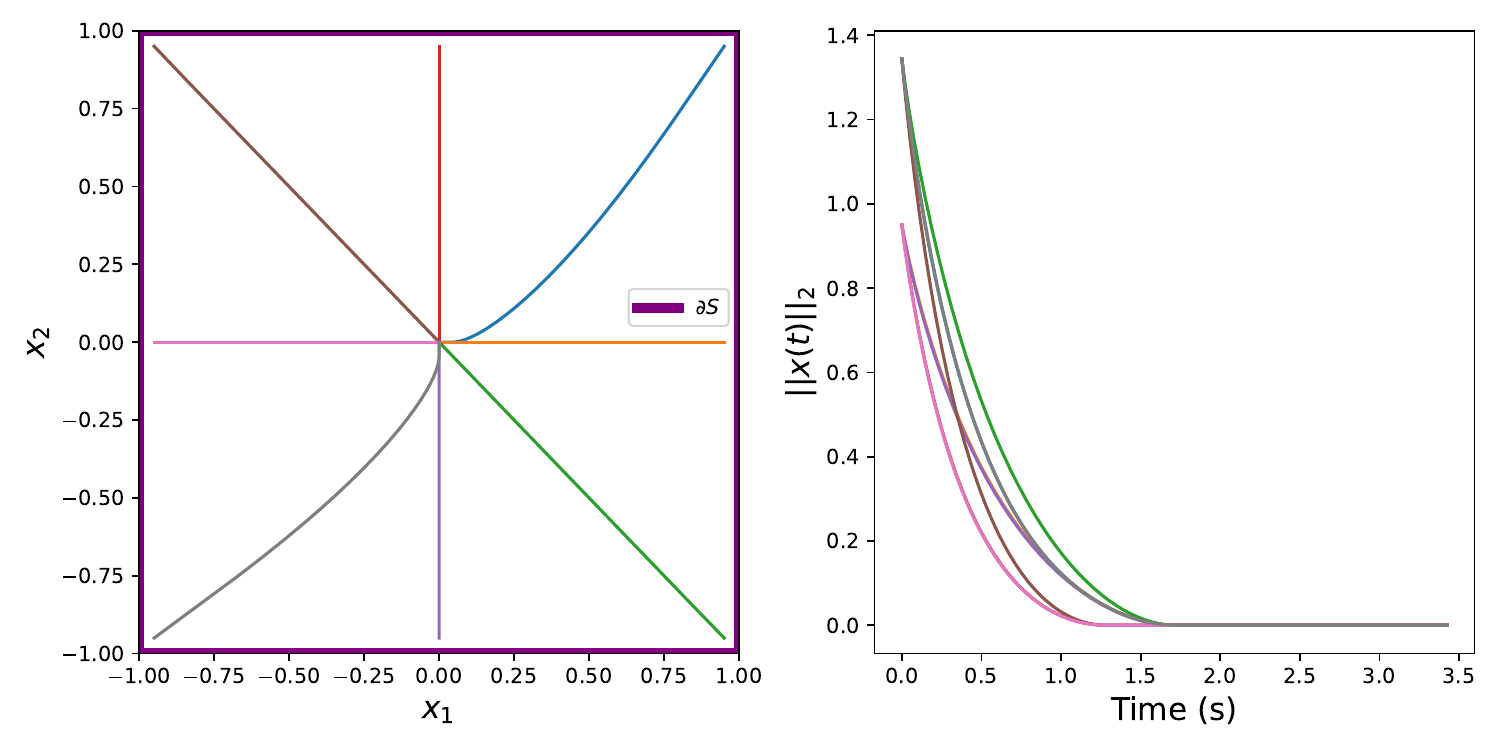}
\caption{\small Robust optimal safe predefined-time stabilization for the bounded safe set. (\textbf{Left}) Closed-loop state trajectories $x(t), \ t \geq 0$, starting from different initial conditions near the boundary $\partial \mathcal{S}$ of the safe set. (\textbf{Right}) Time evolution of the Euclidean norm $\|x(t)\|_2, \ t \geq 0,$ of each trajectory shown in the left plot. In both plots, trajectories starting from the same initial condition are marked with the same color.}
\label{figure:bounded pinn traj}
\end{figure}

Fig.~\ref{figure:bounded pinn appro} shows the approximate Nash value $\hat{V}$, the approximate Nash control strategy $\hat{u}$, and the approximate Nash adversary strategy $\hat{a}$ obtained from Algorithm~\ref{alg:hpinn}, along with the exact Nash value $V$, the exact Nash control strategy $u^\star$, and the exact Nash adversary strategy $a^\star$. To evaluate the approximation accuracy of our algorithm, we define, for every $x \in S_{\text{col}}$, the symmetric absolute error (SAE) for the approximate Nash value $\hat{V}$, the approximate Nash control strategy $\hat{u}$, and the approximate Nash adversary strategy $\hat{a}$ as
\begin{align*}
\text{SAE}\left(\hat{V}(x, w_\mathcal{K}),V(x)\right) & \triangleq \frac{|\hat{V}(x, w_\mathcal{K}) - V(x)|}{|\hat{V}(x,w_\mathcal{K})| + |V(x)|}, \\
   \text{SAE}\left(\hat{u}(x, w_\mathcal{K}),u^\star(x,\theta_\mathrm{c})\right) & \triangleq \frac{\|\hat{u}(x, w_\mathcal{K}) -u^\star(x,\theta_\mathrm{c}) \|_1}{\|\hat{u}(x, w_\mathcal{K})\|_1 + \|u^\star(x,\theta_\mathrm{c}) \|_1} \ , \\ 
\end{align*}
and
\begin{equation*}
    \text{SAE}\left(\hat{a}(x, w_\mathcal{K}), a^\star(x,\theta_\mathrm{a})\right) \triangleq \frac{\|\hat{a}(x, w_\mathcal{K}) -  a^\star(x,\theta_\mathrm{a}) \|_1}{\|\hat{a}(x, w_\mathcal{K})\|_1 + \| a^\star(x,\theta_\mathrm{a}) \|_1}.
\end{equation*} 
As shown in Fig. \ref{figure:bounded pinn appro}, the SAE between the exact and the approximate solution to the game learned by Algorithm \ref{alg:hpinn} is low, verifying the algorithm's efficacy. The left plot of Fig.~\ref{figure:bounded pinn traj} shows the closed-loop state trajectories starting from different initial conditions near the boundary $\partial \mathcal{S}$ of the safe set $\mathcal{S}$. Note that all trajectories remain within $\mathcal{S}$ and converge to the origin. The right plot of Fig.~\ref{figure:bounded pinn traj} shows the time evolution of the Euclidean norm of each trajectory. Note the predefined time convergence to the origin, which confirms the settling-time function bound $T\left(x_0,\theta_\mathrm{c},\theta_\mathrm{a} \right) \leq 3.4259,\ x_0 \in \mathcal{S}$. 

\subsection{Unbounded Safe Set}

 Let $s(x) = 1-x_2^2,\ x \in \mathbb{R}^2,$ and let $V\left(x\right)=\frac{\|x\|_2^2}{2s(x)}, \ x \in \mathcal{S},$ be the Nash value, whereas the terms composing the running cost \eqref{eq: performance integrand} are chosen as
\begin{align*}
L(x)=&\ \frac{1}{2}\left\| \frac{\left\lceil x\right\rfloor^{\gamma_1}}{s^{\frac{\gamma_1-1}{2}}(x)} + \frac{\left\lceil x\right\rfloor^{\gamma_2}}{s^{\frac{\gamma_2-1}{2}}(x)} \right \|_2^2  \\ &+\frac{x_1}{s(x)} \min \{0, x_1\}
+\frac{x_2}{s(x)}\left( 1+ \frac{\|x\|_{2}^{2}}{s(x)}\right) \max \{0, x_2\},\quad x \in \mathcal{S},\\
L_{u}(x)=&L_{a}(x)\\ =&\left( \frac{\left\lceil x\right\rfloor^{\gamma_1}}{s^{\frac{\gamma_1-1}{2}}(x)} + \frac{\left\lceil x\right\rfloor^{\gamma_2}}{s^{\frac{\gamma_2-1}{2}}(x)} \right)^{\mathrm{T}}  
- \left[ \frac{x_1}{s(x)},\ \frac{x_2}{s(x)}\left( 1+ \frac{\|x\|_{2}^{2}}{s(x)}\right)\right],\quad x \in \mathcal{S},
\end{align*}      
$R_u(x)=\frac{1}{4}I_{2},\ x \in \mathcal{S}$, and $R_a(x)=\frac{1}{2}I_{2},\ x \in \mathcal{S}$, where $\gamma_1 \in (0,1)$ and $\gamma_2>1$. Hence, 
with $\theta_\mathrm{c}=\left[\gamma_1,\ \gamma_2 \right]^{\mathrm{T}}$ and $\theta_\mathrm{a}=\left[\gamma_1,\ \gamma_2 \right]^{\mathrm{T}}$, the inverse Nash equilibrium is given by
\begin{align} \label{eq: Nash control}
 u^\star(x,\theta_\mathrm{c})=-2\left( \frac{\left\lceil x\right\rfloor^{\gamma_1}}{s^{\frac{\gamma_1-1}{2}}(x)} + \frac{\left\lceil x\right\rfloor^{\gamma_2}}{s^{\frac{\gamma_2-1}{2}}(x)} \right),\quad x \in \mathcal{S},
\end{align}
and
\begin{align} \label{eq: Nash adversary}
 a^\star(x,\theta_\mathrm{a})=\left( \frac{\left\lceil x\right\rfloor^{\gamma_1}}{s^{\frac{\gamma_1-1}{2}}(x)} + \frac{\left\lceil x\right\rfloor^{\gamma_2}}{s^{\frac{\gamma_2-1}{2}}(x)} \right),\quad x \in \mathcal{S}.
\end{align}
Next, evaluating the time derivative of $V(x)$ along the trajectories of \eqref{eq:lor 1} and \eqref{eq:lor 2} with $u=u^\star(x,\theta_\mathrm{c})$ and $a=a^\star(x,\theta_\mathrm{a})$ given by \eqref{eq: Nash control} and \eqref{eq: Nash adversary} yields
\begin{align} \label{eq: example unbounded vdot} 
\dot{V}(x)= &-\left( \frac{\|x\|_{\gamma_1+1}^{\gamma_1+1}}{s^{\frac{\gamma_1+1}{2}}(x)}  +   \frac{\|x\|_{\gamma_2+1}^{\gamma_2+1}}{s^{\frac{\gamma_2+1}{2}}(x)}\right)\nonumber 
-\frac{1}{s(x)}\left( x_1\min \{0, x_1\}+x_2 \max \{0, x_2\} \right)                       
\nonumber \\  & -\frac{\|x\|_2^2}{ s^2(x)} \Big( \max \{0, x_2\} x_2 
+\frac{|x_{2}|^{\gamma_1+1}}{s^{\frac{\gamma_1-1}{2}}(x)}  +   \frac{|x_{2}|^{\gamma_2+1}}{s^{\frac{\gamma_2-1}{2}}(x)}\Big) \nonumber
\\ \leq & -\left( \frac{\|x\|_{\gamma_1+1}^{\gamma_1+1}}{s^{\frac{\gamma_1+1}{2}}(x)}  +   \frac{\|x\|_{\gamma_2+1}^{\gamma_2+1}}{s^{\frac{\gamma_2+1}{2}}(x)}\right),\quad x \in \mathcal{S}.
\end{align}
Now, by the monotonicity property of H\"{o}lder norms \cite{bernstein2018scalar}, we have $\left\| x\right\|_2 \leq\left\| x \right\|_{\gamma_1+1}$ since $2>\gamma_1+1>0$. Moreover, by the equivalence of vector norms on $\mathbb{R}^{n}$ \cite{bernstein2018scalar}, it follows that $\left\| x \right\|_2 \leq 2^{\frac{\gamma_2-1}{2\left(\gamma_2+1\right)}} \left\|x \right\|_{\gamma_2+1}$ since $2<\gamma_2+1$, and hence, \eqref{eq: example unbounded vdot} can be further bounded as
\begin{align*}
\dot{V}(x) \leq  - 2^{\frac{\gamma_1+1}{2}} V^{\frac{\gamma_1+1}{2}}(x)-2 V^{\frac{\gamma_2+1}{2}}(x), \quad x \in \mathcal{S},
\end{align*}
which implies that \eqref{eqn: lyap4} is satisfied with $\gamma=T_{\mathrm{p}}$, $\alpha=2^{\frac{\gamma_1+1}{2}} ,\ \beta=2,\ p={\frac{\gamma_1+1}{2}},\ q=\frac{\gamma_2+1}{2}$, and $r=1$, and hence, by Corollary \ref{theorem_inverse}, the zero solution $x(t) \equiv 0$ of the closed-loop system is safely predefined-time stable.

Let $T_{\mathrm{p}}=3.4259$ so that $ \gamma_1=0.5$ and $\gamma_2=1.5$. 
For our PINN architecture, we let $h(x) = \frac{1}{1 + e^{-x}}, \ x\in \mathbb{R}$, which is the sigmoid function, and $B(x) = \frac{\|x\|^2_2}{\sqrt{2} - \sqrt{x_2^2+1}},\ x\in \mathcal{S}$. The neural network $V_{\text{NN}}(\cdot,\cdot)$ is implemented as a fully connected neural network with six hidden layers, each comprising $50$ neurons, using the sigmoid activation function. The number of collocation points is set to $M=10^4$. 
For the augmented Lagrangian method, the hyperparameters are set to $\mu_0 = 10^{-3}$ and $\delta = 1.5$. 
At every iteration of Algorithm \ref{alg:hpinn}, the optimization subproblem \eqref{subprob_itek} is solved using the truncated Newton conjugate gradient (TNCG) optimizer (see, for example, \cite{kelley1999iterative} and \cite{bonnans2006numerical}). 

\begin{figure}[!ht] 
\centering\includegraphics[width=0.7\textwidth]{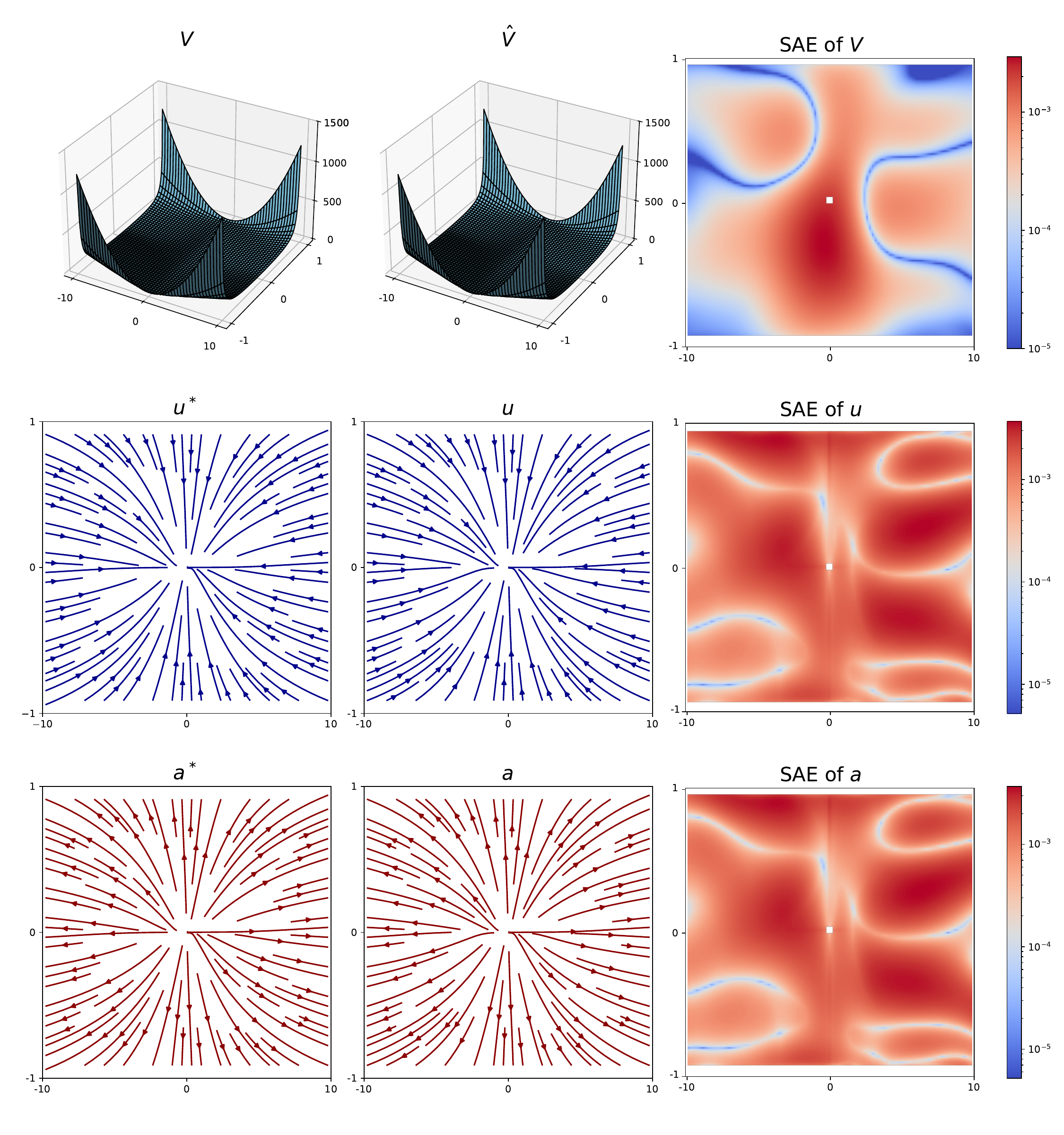}
\caption{\small Nash value, Nash feedback control strategy, and Nash feedback adversary strategy for the unbounded safe set. 
(\textbf{Left}) Exact Nash value $V$, exact Nash control strategy $u^\star$, and exact Nash adversary strategy $a^\star$. 
(\textbf{Middle}) Approximate Nash value $\hat{V}$, approximate Nash control strategy $\hat{u}$, and approximate Nash adversary strategy $\hat{a}$. 
(\textbf{Right}) Symmetric absolute error for learning the Nash value, the Nash control strategy, and the Nash adversary strategy.}

\label{figure:unbounded pinn appro}
\end{figure}
\begin{figure}[!ht]
\centering\includegraphics[width=0.7\textwidth]{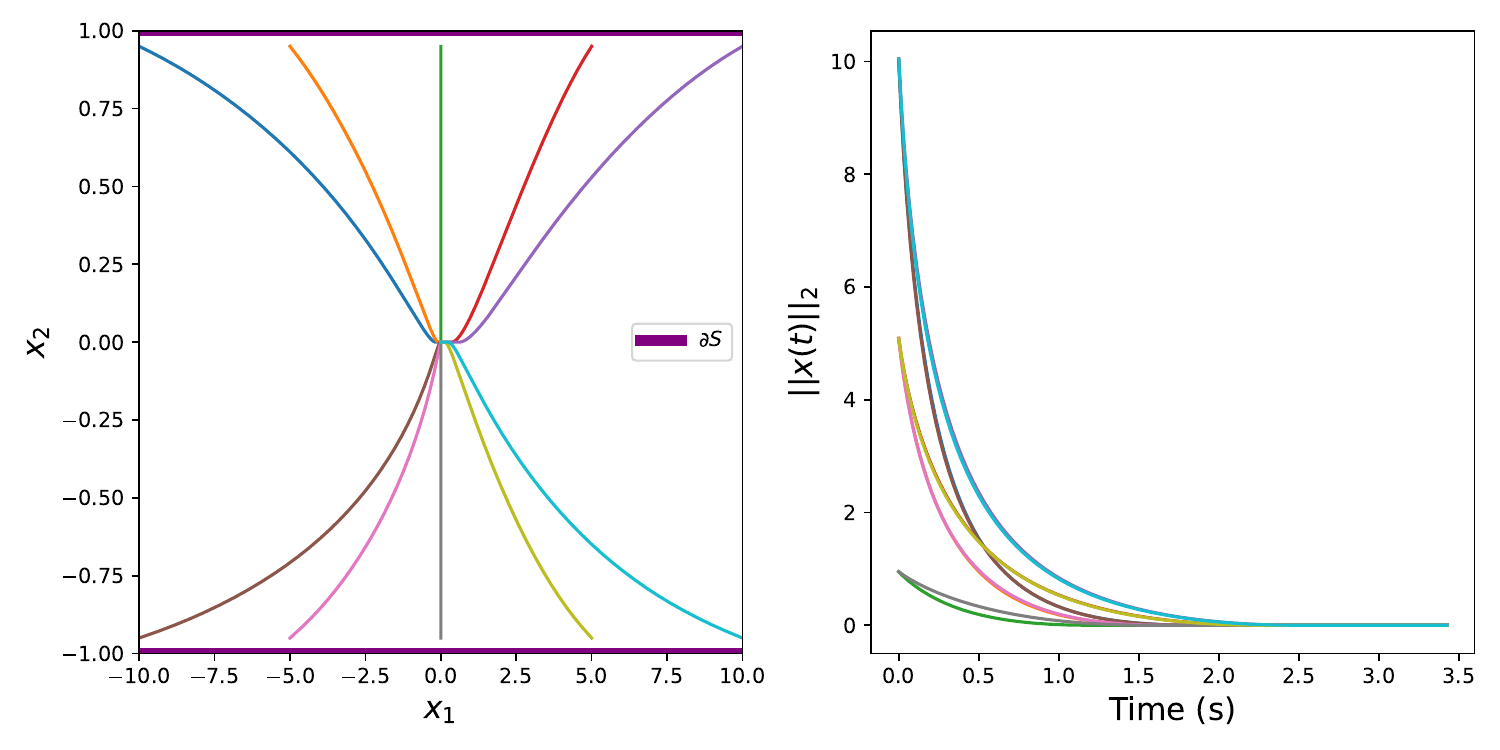}
\caption{\small Robust optimal safe predefined-time stabilization for the unbounded safe set. (\textbf{Left}) Closed-loop state trajectories $x(t),\ t \geq 0$, starting from different initial conditions in the safe set $\mathcal{S}$. (\textbf{Right}) Time evolution of the Euclidean norm $\|x(t)\|_2,\ t \geq 0,$ of each trajectory shown in the left plot. In both plots, trajectories starting from the same initial condition are marked with the same color.}
\label{figure:unbounded pinn traj}
 \end{figure}

Fig.~\ref{figure:unbounded pinn appro} shows the approximate Nash value $\hat{V}$, the approximate Nash control strategy $\hat{u}$, and the approximate Nash adversary strategy $\hat{a}$, along with the exact Nash value $V$, the exact Nash control strategy $u^\star$, and the exact Nash adversary strategy $a^\star$. Note that Algorithm~\ref{alg:hpinn} exhibits a low SAE. The left plot of Fig.~\ref{figure:unbounded pinn traj} shows the closed-loop state trajectories starting from different initial conditions within the safe set $\mathcal{S}$. Note that all trajectories remain within $\mathcal{S}$ and converge to the origin. The right plot of Fig.~\ref{figure:unbounded pinn traj}
 shows the time evolution of the Euclidean norm of each trajectory. Note the predefined time convergence to the origin, which verifies the settling-time function bound $T\left(x_0,\theta_\mathrm{c},\theta_\mathrm{a} \right) \leq 3.4259,\ x_0 \in \mathcal{S}$. 

\section{Conclusion}\label{sec:Future work}

In this paper, an adversarially robust optimal safe predefined-time stabilization problem is stated and shown to correspond to a two-player zero-sum differential game. Sufficient conditions are derived to characterize a safely predefined-time stabilizing saddle-point solution to the differential game. Specifically, safe prede-\ fined-time stability of the closed-loop system and Nash equilibrium are established via a barrier Lyapunov function that satisfies a certain differential inequality and the steady-state HJI equation. Given the intractability of the latter, an adversarial physics-informed learning algorithm is developed to learn the safely predefined-time stabilizing solution to the steady-state HJI equation. Future research will explore discrete-time extensions of the proposed framework.

\bibliographystyle{siamplain}
\bibliography{coordination,proj_bib,multi_uav,references}

\end{document}